\documentclass{amsart}

\usepackage[T1]{fontenc}
\usepackage[utf8]{inputenc}
\usepackage[USenglish]{babel}
\usepackage{textcase}

\usepackage[
	bookmarks=true,
	plainpages=false,
	linktocpage,
	colorlinks=true,
	citecolor=green!80!black,
	linkcolor=red!70!black,
	filecolor=magenta,
	urlcolor=magenta,
	breaklinks,
	pdfauthor={Martin Winter},
]{hyperref}

\usepackage{amsmath,amsthm,amssymb}
\usepackage{calc, mathtools} %calc for \widthof, mathtools for \mathclap

\usepackage{colortbl,color} % colors
\usepackage[dvipsnames]{xcolor}
\usepackage{centernot} % crossing out symbols correctly
\usepackage{array} % arrays
\usepackage{enumitem,moreenum} % numbered lists where each item can be labeled
\usepackage{cite} % for grouping references in the text which are adjacent in the bibliography
\usepackage[nameinlink,capitalize,noabbrev]{cleveref}
\usepackage{nicefrac}

\usepackage{tikz-cd} 

\usepackage[font=small,labelfont=bf]{caption}

\usepackage{blkarray}

\usepackage{listings}% http://ctan.org/pkg/listings
\usepackage{inconsolata}

\newcommand{\RR}{\mathbb{R}}    % real numbers                      
\newcommand{\F}{\mathcal{F}}    %face lattice

\newcommand{\Tsymb}{\top}
\newcommand{\T}{^{\Tsymb}}

\def\^#1{^{(#1)}}
\def\s^#1{^{\smash{(#1)}}}

\newcommand*{\horzbar}{\rule[.5ex]{2.5ex}{0.4pt}}

\def\:{\colon}

\newcommand{\mylabel}{\upshape(\textit{\roman*})}
\newenvironment{myenumerate}{\begin{enumerate}[label=\mylabel]}{\end{enumerate}}

\newcommand{\freespace}{\kern.07em} 
\newcommand{\free}{\freespace\cdot\freespace} 

\newcommand{\enquote}[1]{``#1''}                                 % quotation marks

\newcommand{\ul}[1]{\underline{\smash{#1}}}

\newcommand{\msays}[1]{{\footnotesize\textcolor{red}{#1}}}
\newcommand{\TODO}{\msays{TODO}}

\theoremstyle{plain}  % theorem, lemma, corollary, proposition, conjecture, criterion, algorithm
\newtheorem{theorem}{Theorem}[section]
\newtheorem{corollary}[theorem]{Corollary}
\newtheorem{lemma}[theorem]{Lemma}
\newtheorem{proposition}[theorem]{Proposition}

\theoremstyle{definition} % definition, condition, problem, example
\newtheorem{definition}[theorem]{Definition}
\newtheorem{example}[theorem]{Example}
\newtheorem{remark}[theorem]{Remark}
\newtheorem{question}[theorem]{Question}
\newtheorem{observation}[theorem]{Observation}

\crefname{theorem}{Theorem}{Theorems}
\crefname{proposition}{Proposition}{Propositions}
\crefname{lemma}{Lemma}{Lemmas}
\crefname{corollary}{Corollary}{Corollaries}
\crefname{remark}{Remark}{Remarks}
\crefname{example}{Example}{Examples}
\crefname{definition}{Definition}{Definitions}
\crefname{problem}{Problem}{Problems}
\crefname{observation}{Observation}{Observation}
\crefname{construction}{Construction}{Construction}

\DeclareMathOperator{\conv}{conv}

\DeclareMathOperator{\Aut}{Aut}
\DeclareMathOperator{\diam}{diam}

\DeclareMathOperator{\Ortho}{O}

\DeclareMathOperator{\GL}{GL}

\DeclareMathOperator{\Span}{span}

\DeclareMathOperator{\rank}{rank}

\DeclareMathOperator{\Id}{Id}

\DeclareMathOperator{\dist}{dist}
\DeclareMathOperator{\Eig}{Eig}
\DeclareMathOperator{\Spec}{Spec}
\DeclareMathOperator{\Sym}{Sym}
\DeclareMathOperator{\Perm}{Perm}   	% permutation matrices
   	
\DeclareMathOperator*{\Argmax}{Argmax}   	
\DeclareMathOperator{\vol}{vol}

\let\<=\langle
\let\>=\rangle
\let\x=\times

\def\...{...}
\newcommand{\shortStyle}{\textit}
\newcommand{\ie}{\shortStyle{i.e.,}}

\newcommand{\eg}{\shortStyle{e.g.}}

\newcommand{\wrt}{\shortStyle{w.r.t.}}

\newcommand{\resp}{resp.}

\let\angle=\measuredangle
\makeatletter
\renewcommand*{\eqref}[1]{%
  \hyperref[{#1}]{\textup{\tagform@{\ref*{#1}}}}%
}
\makeatother

\numberwithin{equation}{section}

\newcommand{\tempnewpage}{}

\begin{document}

%%%%%%%%%%%%%%%%%%%%%%%%%%%%%%%%%%%%%%%%%%%%%%%%%%%%%%%%%%%%%%%%%%%%%%%%%%%%%%%%%%%%%%

\expandafter\title%[...]
%{What is ... an eigenpolytope?}
{Eigenpolytopes, spectral polytopes and edge-transitivity}
		
\author[M. Winter]{Martin Winter}
\address{Faculty of Mathematics, University of Technology, 09107 Chemnitz, Germany}
\email{martin.winter@mathematik.tu-chemnitz.de\newline\rule{0pt}{1.5cm}\includegraphics[scale=0.7]{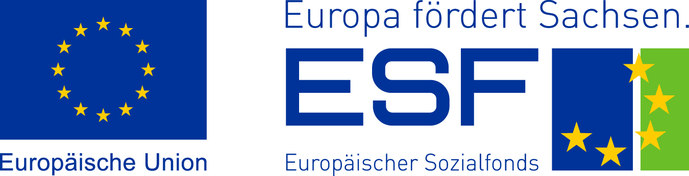}}
	
% 51M20 - Polyhedra and polytopes
% 52B05 - Combinatorial properties of polytopes and polyhedra
% 52B11 - $n$-dimensional polytopes
% 52B12 - Special polytopes
% 52B15 - Symmetry properties of polytopes
% 05C50 - Graphs and linear algebra
% 05C62 - Graph representations
		
\date{\today}
\begin{abstract}
Starting from a finite simple graph $G$, for each eigenvalue $\theta$ of its adjacency matrix 
%\mbox{$\theta\in\Spec(G)$} 
one can construct a convex polytope $P_G(\theta)$, the so called \emph{$\theta$-eigenpolytop} of $G$.
For some polytopes this technique can be used to reconstruct the polytopes from its edge-graph.
Such polytopes (we shall call them \emph{spectral}) are still badly understood.
%Occasionally, eigenpolytopes can be used to reconstruct a polytope from its edge-graph, but it is not well-understood for which graphs/polytopes this works.
%We call a graph/polytope \emph{spectral} if this works .
We give an overview of the literature for eigenpolytopes and spectral polytopes.

We introduce a geometric condition by which to prove that a given polytope is \mbox{spectral} (more exactly, $\theta_2$-spectral).
We apply this criterion to the \emph{edge-transitive} polytopes.
We~show that every edge-transitive polytope is $\theta_2$-spectral, is uniquely determined by this graph, and realizes all its symmetries.
We give a complete classification of distance-transitive polytopes.
\end{abstract}

\maketitle

\section{Introduction}
\label{sec:introduction}

\emph{Eigenpolytopes} are a construction in the intersection of combinatorics and geometry, using techniques from spectral graph theory.
Eigenpolytopes provide a way~to associate several polytopes to a finite simple graph, one for each eigenvalues of~its adjacency matrix.
A formal definition can be found in \cref{sec:def_eigenpolytope}.

\vspace{0.5em}
\begin{center}
\includegraphics[width=0.7\textwidth]{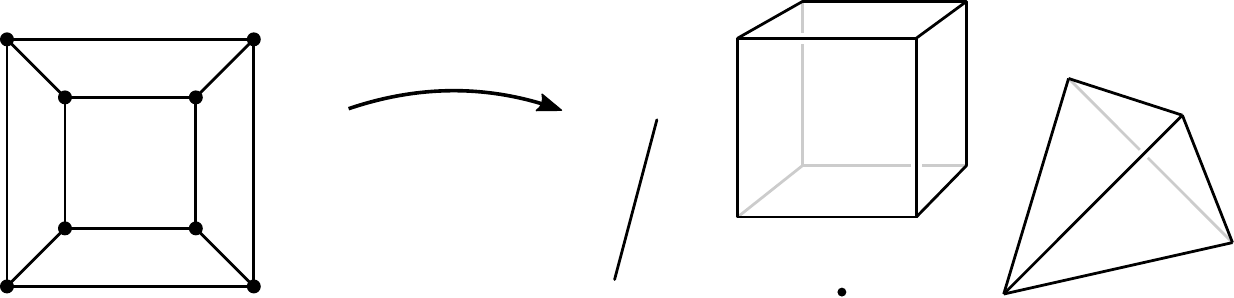}
\end{center}
\vspace{0.5em}

Eigenpolytopes can be applied from two directions:
for the first, one starts~from a given graph, computes its eigenpolytopes, and tries to deduce, from the geometry and combinatorics of these polytopes, something about the original~graph.
For the other direction, one starts with a polytope, asks whether it is an eigenpolytope, and if so, for which graphs, which eigenvalues, and how these relate to the original polytope.
Eigenpolytopes have several interesting geometric and algebraic properties, and establishing that a family of polytopes consists of eigenpolytopes opens up their study to the techniques of spectral graph theory.

For some graphs the connection to their eigenpolytopes is especially strong: it can happen  that a graph is the edge-graph of one of its eigenpolytopes, or equivalently, that a polytope is an eigenpolytope of its edge-graph.
Such graphs/polytopes are quite special and we shall call them \emph{spectral}.
For example, all regular polytopes are spectral, but there are many others.
Their properties are not well-understood.

We survey the literature of eigenpolytope and spectral polytopes.
We establish a technique with which to prove that certain polytopes are spectral polytopes and we apply it to \emph{edge-transitive} polytopes. That are polytopes for which the Euclidean symmetry group $\Aut(P)\subset\Ortho(\RR^d)$ acts transitively on the set of edge $\F_1(P)$.
As we shall explain, this characterization suffices to proves that an edge-transitive polytope is uniquely determined by its edge-graph, and also realizes all its combinatorial symmetries.
A complete classification of edge-transitive polytopes is not known as of yet.
However, using results on eigenpolytopes, we are able to give a complete classification of a sub-class of the edge-transitive polytopes, namely, the \emph{distance-transitive polytopes}.

%Besides this results, we aim to give an extensive literature overview of eigenpolytopes, including the first perspective \enquote{from the graph to the polytope}.
%A phenomenon that seems to haev attracted peoples attention repeatedly, is when a polytope is the eigenpolytope of its edge-graph.
%We shall call these polytopes \emph{spectral}.
%...
%
%The prerequisite for this paper is a mild dose of spectral graph theory.
%It might be helpful to already have heard about spectral realizations, but we also get away without mentioning this concept explicitly.
%In case of doubt, consider \cite{winter2020symmetric}, where we previously discussed spectral realizations of highly symmetric graphs.
%The proofs of some statements will be just cited from this paper.

\subsection{Outline of the paper}

\cref{sec:eigenpolytopes} starts with a motivating example for directing the reader towards the definition of the \emph{eigenpolytope} as well as the phenomenon of \emph{spectral} graphs and polytopes.
We include a literature overview for~\mbox{eigenpolytopes} and spectral polytopes.

In \cref{sec:balanced_spectral} we give a first rigorous definition for the notion \enquote{spectral polytope} via \emph{balanced polytopes}. The latter is a notion related to the rigidity theory.
%This highlights the rigidity theory perspective of that idea.

In \cref{sec:izmestiev} we introduce the, as of yet, most powerful tool for proving that~certain polytopes are spectral.

In the final section, \cref{sec:edge_transitive}, we apply this result to \emph{edge-transitive} polytopes.
It is a simple corollary of the previous section that these are $\theta_2$-spectral.
We explore the implications of this finding: edge-transitive polytopes (in dimension $d\ge 4$) are uniquely determined by the edge-graph and realize all of its symmetries.~%
We~discuss sub-classes, such as the arc-, half- and distance-transitive polytopes.
We close with a complete classification of the latter (based on a result of Godsil).

%The idea itself has spawned several times and appears too immediate to not yield any interesting results.
%
%In easy terms, one starts with a finite, simple graph $G$ and chooses an eigenvalue $\theta\in\Spec(G)$ (\ie\ of its adjacency or Laplace matrix).
%From this, one constructs a polytope $P_G(\theta)$, the so-called \emph{$\theta$-eigenpolytope}.
%If the eigenvalue had multiplicity $d$, then the polytope has dimension $d$.
%One hence expects the most interesting results for graphs with large eigenspaces.
%
%The hope then is that the geometry and combinatorics of $P_G(\theta)$ tells us something deep about $G$.
%Results of this form are not too abundant.
%Still, people have asked other questions about eigenpolytope, and the answeres turn out to be interesting.

\tempnewpage

\section{Eigenpolytopes and spectral polytopes}
\label{sec:eigenpolytopes}

\subsection{A motivating example}
\label{sec:example}

Let $G=(V,E)$ be the edge-graph of the cube, with vertex set $V=\{1,...,8\}$,~num\-bers assigned to the vertices as in the figure below.
\begin{figure}[h!]
\centering
\includegraphics[width=0.48\textwidth]{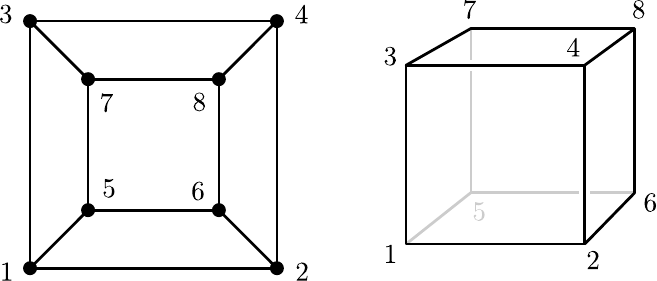}
\end{figure}

\noindent
The spectrum of that graph (\ie\ of its adjacency matrix) is $\{(-3)^1,(-1)^3,1^3,3^1\}$.
Most often, one denotes the largest eigenvalue by $\theta_1$, the second-largest by $\theta_2$, and so on.
In spectral graph theory, there exists the general rule of thumb that the most exciting eigenvalue of a graph is not its largest, but its \emph{second-largest} eigenvalue $\theta_2$ (which is related to the \emph{algebraic connectivity} of $G$).

For the edge-graph of the cube, we have $\theta_2=1$,~of multiplicity \emph{three}.
And here are three linearly independent eigenvectors to $\theta_2$:
%
%\begin{align*}
%u_1^\top &= (+ 1,+ 1,+ 1,+ 1,-1,-1,-1,-1), \\
%u_2^\top &= (+ 1,+ 1,- 1,- 1,+1,+1,-1,-1), \\
%u_3^\top &= (+ 1,- 1,+ 1,- 1,+1,-1,+1,-1), \\
%\end{align*}
%

\vspace{0.5em}

\begin{center}
\raisebox{-3.3em}{\includegraphics[width=0.23\textwidth]{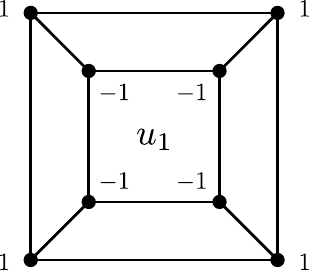}}
\qquad
$
u_1 = \begin{pmatrix}
\phantom+1\\ \phantom+1 \\ \phantom+1 \\ \phantom+1 \\ -1 \\ -1 \\ -1 \\ -1
\end{pmatrix},\quad
u_2 = \begin{pmatrix}
\phantom+1\\ \phantom+1 \\ -1 \\ -1 \\ \phantom+1 \\ \phantom+1 \\ -1 \\ -1
\end{pmatrix},\quad
u_3 = \begin{pmatrix}
\phantom+1\\ -1 \\ \phantom+1 \\ -1 \\ \phantom+1 \\ -1 \\ \phantom+1 \\ -1
\end{pmatrix}.
$
\end{center}

\vspace{0.7em}
\noindent
We can write these more compactly in a single matrix $\Phi\in\RR^{8\x 3}$:

\vspace{0.4em}
\begin{center}
$\Phi= \;\begin{blockarray}{(lll)r}%\begin{block}{(ccc)c}
\phantom+1 & \phantom+1 & \phantom+1\;\; &\text{\quad \footnotesize $\leftarrow v_1$}\\
\phantom+1 & \phantom+1 & -1 & \text{\quad \footnotesize $\leftarrow v_2$} \\
\phantom+1 & -1 & \phantom+1 & \text{\quad \footnotesize $\leftarrow v_3$} \\
\phantom+1 & -1 & -1 & \text{\quad \footnotesize $\leftarrow v_4$} \\
-1 & \phantom+1 & \phantom+1 & \text{\quad \footnotesize $\leftarrow v_5$} \\
-1 & \phantom+1 & -1 & \text{\quad \footnotesize $\leftarrow v_6$} \\
-1 & -1 & \phantom+1 & \text{\quad \footnotesize $\leftarrow v_7$} \\
-1 & -1 & -1 & \text{\quad \footnotesize $\leftarrow v_8$} \\
%\end{block}
\end{blockarray}.$
\qquad
\raisebox{-4.2em}{\includegraphics[width=0.4\textwidth]{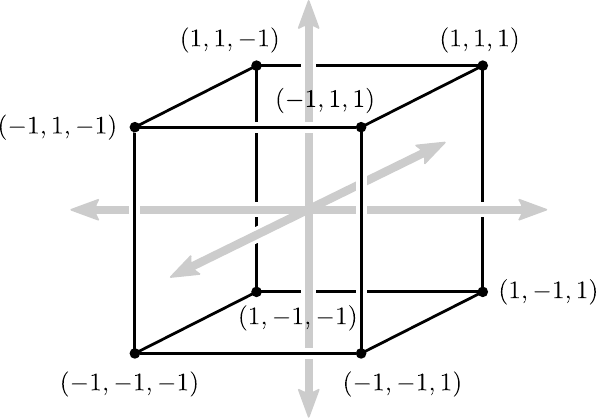}}
\end{center}

We now take a look at the rows of that matrix, of which it has exactly eight.~These rows are naturally assigned to the vertices of $G$ (assign $i\in V$ to the $i$-th row of~$\Phi$), and each row can be interpreted as~a~vec\-tor in $\RR^3$.

If we place each vertex $i\in V$ at the position $v_i\in\RR^3$ given by the $i$-th row of $\Phi$, we find that this embedds the graph $G$ \emph{exactly} as the skeleton of a cube (see the figure above).
In other words: if we compute the convex hull of the $v_i$, we get back the polyhedron from which we have started.
What a coincidence, isn't it?

This example was specifically chosen for its nice numbers, but in~fact,~the same works out as well for many other polytopes, inclu\-ding all the regular polytopes~in all dimension.
One probably learns to appreciate this magic when suddenly in~need for the vertex coordinates of some not so nice polytope, say, the regular dodecahedron or 120-cell.
With this technique in the toolbox, these coordinates are just one eigenvector-computation away (we included a short Mathematica script in \cref{sec:appendix_mathematica}).
Note also, that we never specified~the~dim\-en\-sion of~the~embed\-ding, but it just so happened, that the second-largest eigenvalue has the right multiplicity.
This phenomenon definitely deserves an explanation.

%Now, a short reality check: could this have been a coincidence?
%In fact, \emph{no}.
%The same works perfectly well for all the other regular polytopes (always making use of the second-largest eigenvalue).

\subsubsection*{On the choice of eigenvectors}
\label{sec:choice_of_eigenvectors}

One might object that the chosen eigenvectors $u_1, u_2$ and $u_3$ look suspiciously cherry-picked, and we may not get such a nice result if we would have chosen just any eigenvectors.
And this is true.
For an appropriate choice of these vectors, we can, instead of a cube, get a cuboid, or a parallelepiped.
In fact, we can obtain any \emph{linear} transformations of the cube.
\emph{But}, we can also get \emph{only} linear transformations, and nothing else.
The reason is the following well~know  fact from linear algebra:

\begin{theorem}
\label{res:same_column_span}
Two matrices $\Phi,\Psi\in\RR^{n\x d}$ have the same column span, \ie~$\Span \Phi=\Span \Psi$, if and only if their rows are related by an invertible linear transformation, \ie\ $\Phi=\Psi T$ for some $T\in\GL(\RR^d)$.
\end{theorem}

\noindent
%Since the rows of our matrices are the coordinates of the $v_i$, this explains our the comment concerning linear transformations.
In our case, the column span is the $\theta_2$-eigenspace, and the rows are the coordinates of the $v_i$.
We say that any two polytopes constructed in this way are \emph{linearly equivalent}.

The only notable property of the chosen basis in the example is, that the vectors $u_1, u_2$ and $u_3$ are orthogonal and of the same length.
Any other choice of such a basis of~the~\mbox{$\theta_2$-eigen}\-space (\eg\ an orthonormal basis) would also have given a cube, but reoriented, rescaled and probably with less nice coordinates.
For details on how this choice relates to the orientation, see \eg\ \cite[Theorem 3.2]{winter2019geometry}.

\subsection{Eigenpolytopes}
\label{sec:def_eigenpolytope}

We compile our example into a definition.

\begin{definition}\label{def:eigenpolytope}
Start with a graph $G=(V,E)$, an eigenvalue $\theta\in\Spec(G)$ thereof, as well as an orthonormal basis $\{u_1,...,u_d\}\subset\RR^n$  of the $\theta$-eigenspace. 
We define the \emph{eigenpolytope matrix} $\Phi\in\RR^{n\x d}$ as the matrix in which the $u_i$ are the columns:\phantom{mm}
\begin{equation}
\label{eq:eigenpolytope_matrix}
\Phi :=\begin{pmatrix}
		\mid & & \mid \\
		u_1 & \!\!\cdots\!\!\! & u_d \\
		\mid & & \mid		
	\end{pmatrix}=
\begin{pmatrix}
	\;\horzbar\!\!\!\! & v_1\T & \!\!\!\!\horzbar\;\; \\
	& \vdots & \\[0.4ex]
	\;\horzbar\!\!\!\! & v_n\T & \!\!\!\!\horzbar\;\;
\end{pmatrix}.
\end{equation}
Let $v_i\in\RR^d$ denote the $i$-th row of $\Phi$.
The polytope
$$P_G(\theta):=\conv\{v_i\mid i\in V\}\subset\RR^d$$
is called \emph{$\theta$-eigenpolytope} (or just \emph{eigenpolytope}) of $G$.
\end{definition}

For later use we define the \emph{eigenpolytope map} 
\begin{equation}
\label{eq:eigenpolytope_map}
\phi:V\ni i\mapsto v_i\in\RR^d
\end{equation}
that to each vertex $i\in V$ assignes the $i$-th row of the eigenpolytope matrix.

Note that the basis $\{u_1,...,u_d\}\subset\Eig_G(\theta)$ in \cref{def:eigenpolytope} is explicitly chosen~to be an \emph{orthonormal basis}.
This is not strictly necessarily, but this choice is convenient from a geometric point of view: 
a different choice for this basis gives the same~poly\-tope, but with a different orientation rather than, say, transformed~by~a~general linear transformation.
This preserves metric properties and is closer to how polytopes are usually consider up to rigid motions.
We can also reasonably speak of \emph{the} $\theta$-eigenpolytope, as any two differ only by orientation.
%Symmetries of polytopes are often considered as rigid motions.

%Note again that the choice of the basis $u_1,...,u_d\in\RR^d$ only controls the orientation of the eigenpolytope, but has no influence on its combinatorial type, etc.

%More generally, to any polytope $P\subset\RR^d$ with vertices labeled $v_1,...,v_n\in \F_0(P)$, we can associate a matrix $M$, in which the $v_i$ are the rows.
%We shall call this matrix the \emph{arrangement matrix} of $P$.
%Because \cref{def:eigenpolytope} chooses an \emph{orthonormal} basis of eigenvectors, we have the special case $M\T M=\Id$.

With this terminology in place, our observation in the example of \cref{sec:example}~can be summarized as \enquote{the cube is the $\theta_2$-eigenpolytope of its edge-graph}, or alternatively as \enquote{the cube-graph is the edge-graph of its $\theta_2$-eigenpolytope}.
Here is~a~depic\-tion of all the eigenpolytopes of the cube-graph, one for each eigenvalue:

\vspace{0.5em}
\begin{center}
\includegraphics[width=0.7\textwidth]{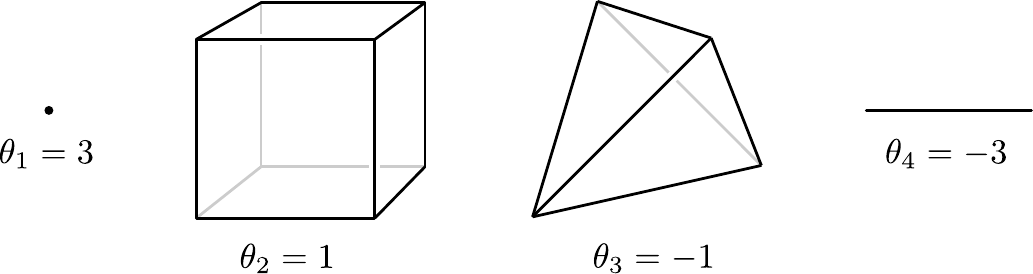}
\end{center}

\noindent
We observe that the phenomenon from \cref{sec:example} only happens for $\theta_2$.
In general, the $\theta_1$-eigenpolytope of a regular graph will always be a single point (which is, why we rarely care about the largest eigenvalue).
Also, whenever a graph is bipartite, the eigenpolytope to the smallest eigenvalue is 1-dimensional, hence a line segment.

We are now free to compute the eigenpolytopes of all kinds of graphs, \mbox{including} graphs which are not the edge-graph of any polytope (so-called \emph{non-polytopal} graphs).
It is then little surprising that no edge-graph of any of its eigenpolytope gives the original graph again.

But even if we start from a polytopal graph, one is not guaranteed to find an eigen\-polytope that has the initial graph as its edge-graph (\eg\ the edge-graph of the triangular prism has no eigenvalue of multiplicity three, hence no eigenpolytope of dimension three, see also \cref{ex:prism}).
Equivalently, if one starts with a polytope, it~is~not guaranteed that this polytope is the eigenpolytope of its edge-graph (or even combinatorially equivalent to it).

\begin{example}
\label{ex:neighborly_1}
A \emph{neighorly polytope} is a polytope whose edge-graph is the complete graph $K_n$.
The spectrum of $K_n$ is $\{(-1)^{n-1},(n-1)^1\}$.
One checks that~the~eigenpolytopes are a single point (for $\theta_1=n-1$) and the regular simplex of dimension $n-1$ (for $\theta_2=-1$).

Consequently, no neighborly polytope other than a simplex is combinatorially equivalent to an eigenpolytope of its edge-graph.
\end{example}

That a graph and its eigenpolytope translate into each other as well as in the case of the cube in \cref{sec:example} is a very special phenomenon, to which we shall give a name:
%Both situations are very special.
a polytope (or graph) for which this happens, will be called \emph{spectral}\footnote{There was at least one previous attempt to give a name to this phenomenon, namely, in \cite{licata1986surprising}, where it was called \emph{self-reproducing}.}.
We cannot formalize this definition right away, as there is some subtlety we have to discuss first (we give a formal definition in \cref{sec:balanced_spectral}, see \cref{def:spectral}).

\begin{example}
\label{ex:pentagon}
The image below shows two spectral realizations of the 5-cycle $C_5$\footnote{Spectral realizations are essentially defined like eigenpolytopes, assinging coordinates $v_i\in\RR^d$ to each vertex $i\in V$ (as in \cref{def:eigenpolytope}), but without taking the convex hull. Instead, one draws the edges between adjacent vertices.}.
\begin{center}
\includegraphics[width=0.4\textwidth]{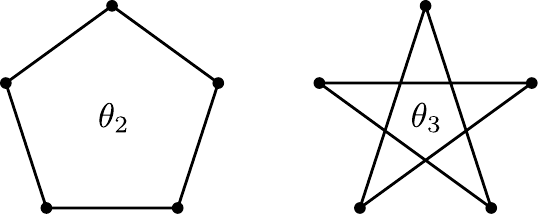}
\end{center}
The left image~shows the realization to the second-largest eigenvalue $\theta_2$, the right image shows the realization to the smallest eigenvalue $\theta_3$.
In both cases, the convex hull (the actual eigenpolytope) is a regular pentagon, whose edge-graph is $C_5$ again.
But we see that only in the case of $\theta_2$ the edges of the graphs get properly mapped into the edges of the pentagon.

While it is true that the 5-cycle $C_5$ is the edge-graph of its $\theta_3$-eigenpolytope, the adjacency informations gets scrambled in the process:
while, say, vertex 1 and 2 are adjacent in $C_5$, their images $v_1$ and $v_2$ do not form an edge in the $\theta_3$-eigenpolytope.
We do not want to call this \enquote{spectral}, as the adjacency information is not preserved.

The same can happen in higher dimensions too, \eg\ with $G$ being the edge-graph of the dodecahedron:
\begin{center}
\includegraphics[width=0.55\textwidth]{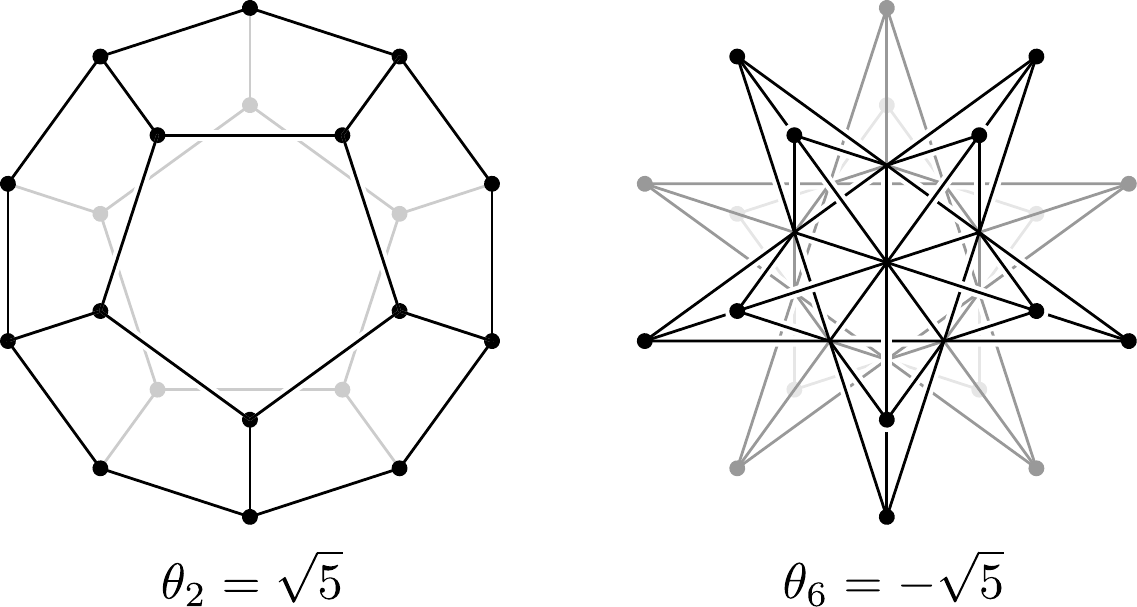}
\end{center}

\end{example}

%\begin{definition}
%Let $P\subset\RR^d$ be a polytope, and $G=(V,E)$ a graph.
%\begin{myenumerate}
%\item $P$ is said to be  \emph{$\theta$-spectral} (or just \emph{spectral}) if it is the $\theta$-eigenpolytope of its edge-graph (up to scale and orientation)\footnote{One could think of a weaker version where we require that the polytope is only \emph{combinatorially equivalent} to the eigenpolytope if its edge-graph, but we make not use of this in this paper.}.
%\item $G$ is said to be \emph{$\theta$-spectral} (or just \emph{spectral}) if it is the edge-graph~of~its~$\theta$-eigenpolytope.
%\end{myenumerate}
%\end{definition}

\begin{observation}
\label{res:theta_2_observations}
From studying many examples, there are two interesting observations to be made, both concern $\theta_2$, none of which is rigorously proven: %, and both involve the second-largest eigenvalue $\theta_2$.
\begin{myenumerate}
\item
It appears as if only $\theta_2$ can give rise to spectral polytopes/graphs.
At~least, all known examples are $\theta_2$-spectral (see also \cref{q:not_theta_2}). Some considerations on nodal domains make this plausible, but no proof is known in the general case (a proof is known in certain special cases, see \cref{res:edge_transitive_spectral_graph}).
\item
If $i\in V$ is a vertex of $G$, then $v_i$ is not necessarily a vertex of every~eigenpolytope ($v_i$ might end up in the interior of $P_G(\theta)$ or one of its faces).
And even if $v_i,v_j\in\F_0(P_G(\theta))$ are distinct vertices and $ij\in E$ is an edge of $G$, it is still not necessarily true that $\conv\{v_i,v_j\}$ is also an edge of the~eigen\-polytope (as seen in \cref{ex:pentagon}).

However, this seems to be no concern in the case $\theta_2$.
It appears as if all edges of $G$ become edges of the $\theta_2$-eigenpolytope, even if $G$ is not spectral (under mild assumptions on the end vertices of the edge).
In other words, the adjacency information of $G$ gets imprinted on the edge-graph of the $\theta_2$-eigenpolytope, whether $G$ is spectral or not.
This is known to be true only in the case of distance-regular graphs \cite[Theorem 3.3 (b)]{godsil1998eigenpolytopes}, but unproven~in general (see also \cref{q:realizing_edges})
\end{myenumerate}

\end{observation}

%
%\begin{itemize}
%	\item Whenever a polytope is spectral, then it is $\theta_2$-spectral.
%	In other words, this interesting coincidence never happens for any eigenvalue other than $\theta_2$.
%	\item
%	Each edge of a graph gives an edge in the $\theta_2$-eigenpolytope.
%	That is, if $ij\in E(G)$, then $\conv\{v_i,v_j\}$ is an edge of $P_G(\theta_2)$.
%\end{itemize}

%The eigenpolytope construction can be approached from two direction:
%first, one already has a graph, and wonders what polytopes will emerge as its eigenpolytopes and what this tells us about the graph.
%Second, we are given a polytope, and we wonder whether it is the eigenpolytope of a graph.

%
%\begin{theorem}
%Every combinatorial symmetry of $G$ is a Euclidean symmetry of any eigenpolytope.
%More precisely, if $\sigma\in\Aut(G)\subseteq\Sym(V)$ is a symmetry of $G$, and $\Pi_\sigma\in\Perm(\RR^n)$ is the associated permutation matrix, then
%%
%$$T_\sigma:=M\T ...$$
%\end{theorem}

\subsection{Litarture}

Eigenpolytope were first introduced by Godsil \cite{godsil1978graphs} in 1978.
Godsil proved the existence of a group homomorphism $\Aut(G)\to\Aut(P_G(\theta))$, \ie\, any combinatorial symmetry of the graph translates into a Euclidean symmetry of the polytope.
From that, he deduces results about the combinatorial symmetry group of the original graph.

We say more about the group homomorphism: for every $\theta\in\Spec(G)$ we have
\begin{theorem}[\!\cite{godsil1978graphs}, Theorem 2.2]
\label{res:realizing_symmetries}
%Let $\theta\in\Spec(G)$ be an eigenavlue of $G$ of multiplicity $d$, and~$\Phi\in\RR^{n\x d}$ be the matrix defined in \eqref{eq:eigenpolytope_matrix}.
If $\sigma\in\Aut(G)\subseteq\Sym(n)$ is a symmetry of $G$, and $\Pi_\sigma\in\Perm(\RR^n)$ is the associated permutation matrix, then
$$T_\sigma:=\Phi\T \Pi_\sigma \Phi \;\in\; \Ortho(\RR^d),\qquad (\text{$\Phi$ is the eigenpolytope matrix})$$
is a Euclidean symmetry of the eigenpolytope $P_G(\theta)$ that also permutes the $v_i$ as~prescribed by $\sigma$, \ie\ $T_\sigma\circ \phi = \phi\circ\sigma$, or $T_\sigma v_i = v_{\sigma(i)}$ for all $i\in V$.
%
%\begin{proof}
%\TODO
%\end{proof}
\end{theorem}

This result is also proven (more generally for spectral graph realizations) in \cite[Corollary 2.9]{winter2020symmetric}.

\cref{res:realizing_symmetries} explicitly uses that eigenpolytopes are defined using an  \emph{orthonormal} bases rather than any basis of the eigenspace, to conclude that the symmetries $T_\sigma$ are \emph{orthogonal} matrices.
%From that, Godsil deduced results about the combinatorial symmetry group of the original graph.
Also, the statement of \cref{res:realizing_symmetries} is not too satisfying in general, as it can happen that non-trivial~symmetries of $G$ are mapped to the identity transformation.
We not necessarily have $\Aut(G)\cong\Aut(P_G(\theta))$.

Several authors construct the eigenpolytopes of certain famous graphs or graph families.
Powers \cite{powers1986petersen} computed the eigenpolytopes of the \emph{Petersen graph}, which he termed the \emph{Petersen polytopes} (one of which will appear as a distance-transitive polytope in \cref{sec:distance_transitive}).
The same author also investigates eigenpolytopes of general distance-regular graphs in \cite{powers1988eigenvectors}.
In \cite{mohri1997theta_1}, Mohri described the face structure of the \emph{Hamming polytopes}, the $\theta_2$-eigenpolytopes of the Hamming graphs.
Seemingly unknown to the author, these polytopes can also by described as the cartesian powers of regular simplices (also distance transitive, see \cref{sec:distance_transitive}).

There exists a wonderful enumeration of the eigenpolytopes (actually, spectral realizations) of the edge-graphs of all uniform polyhedra in \cite{blueSpectral}. Sadly, this write-up was never published formally.
This provides empirical evidence that every uniform polyhedron % (including the Platonic solids, Archimedean solids, prisms and anti-prisms)
has a spectral realization.
The same question might then be asked for uniform polytopes in higher dimensions.

Rooney \cite{rooney2014spectral} used the combinatorial structure of the eigenpolytope (the size of their facets) to deduce statements about the size of cocliques in a graph.

In \cite{padrol2010graph}, the authors investigates how common graph operations translate to operations on their eigenpolytopes.

Particular attention was given to the eigenpolytopes of distance-regular graphs \cite{powers1988eigenvectors,godsil1998eigenpolytopes,godsil1995euclidean}.
It was shown that in a $\theta_2$-eigenpolytope of a distance-regular graph~$G$,~every edge of $G$ corresponds to an edge of the eigenpolytope \cite{godsil1998eigenpolytopes}.
Consequently, $G$~is a spanning subgraph of the edge-graph of the eigenpolytope.
It remains open if the same holds for less regular graphs, \eg\ 1-walk regular graphs or arc-transitive graphs (see also \cref{q:realizing_edges}).

The observation that some polytopes are the eigenpolytopes %(actually, $\theta_2$-eigen\-polytopes) 
of their edge-graph (\ie\ they are \emph{spectral} in our terminology) was made repeatedly, \eg\ in \cite{godsil1995euclidean} and \cite{licata1986surprising}. 
In the latter, this was shown for all regular polytopes, excluding the exceptional 4-dimensional polytopes, the 24-cell, 120-cell and 600-cell.
This gap was filled in \cite{winter2020symmetric} via general considerations concerning spectral realizations of arc-transitive graphs.
In sum, all regular polytopes are known to be $\theta_2$-spectral.
%In \cite{licata1986surprising} numerical experiments have been undertaken that suggest that other polytopes are $\theta_2$-spectral.

The next major result for spectral polytopes was obtained by Godsil in \cite{godsil1998eigenpolytopes}, where he was able to classify all $\theta_2$-spectral distance-regular graphs (see also \cref{sec:distance_transitive}):

\begin{theorem}[\!\cite{godsil1998eigenpolytopes}, Theorem 4.3]
\label{res:spectral_distance_regular_graphs}
Let $G$ be distance-regular. 
If $G$ is $\theta_2$-spectral,~then $G$ is one of the following:
\begin{enumerate}[label=$(\text{\roman*}\,)$]
	\item a cycle graph $C_n,n\ge 3$,
	\item the edge-graph of the dodecahedron,
	\item the edge-graph of the icosahedron,
	\item the complement of a disjoint union of edges,
	\item a Johnson graph $J(n,k)$,
	\item a Hamming graph $H(d,q)$,
	\item a halved $n$-cube $\nicefrac12 Q_n$,
	\item the Schläfli graph, or
	\item the Gosset graph.
\end{enumerate}
\end{theorem}

A second look at this list reveals a remarkable \enquote{coincidence}: while the generic distance-regular graph has few or no symmetries, all the graphs in this list are highly symmetric, in fact, \emph{distance-transitive} (a definition will be given in \cref{sec:distance_transitive}).

It is a widely open question whether being spectral is a property solely reserved for highly symmetric graphs and polytopes (see also \cref{q:trivial_symmetry}).
There is only a single known spectral polytope that is not vertex-transitive (see also \cref{rem:edge_not_vertex} and \cref{q:spectral_non_vertex_transitive}).

\section{Balanced and spectral polytopes}
\label{sec:balanced_spectral}

In this section we give a second  approach to \emph{spectral polytopes} that circumvents the mentioned subtleties.

For the rest of the paper, let $P\subset\RR^d$ denote a full-dimensional polytope in~dimen\-sion $d\ge 2$ with vertices $v_1,...,v_n\in\F_0(P)$.
We disinguish the \emph{skeleton} of $P$, which is the graph with vertex set $\F_0(P)$ and edge set $\F_1(P)$, from the~\mbox{\emph{edge-graph}}~$G_P=(V,E)$ of $P$, which is isomor\-phic to the skeleton, but has vertex set $V=\{1,...,n\}$. The isomorphism will be denoted
\begin{equation}
\label{eq:vertex_map}
\psi:V\ni i\mapsto v_i\in\F_0(P),
\end{equation}
and we call it the \emph{skeleton map}.

% and $G_P=(V,E)$ its edge-graph.
%We shall assume that $G_P$ has vertex~set $V=\{1,...,n\}$, and so has so be distinguished from the \emph{skeleton} of $P$, whose vertices and edges are $\F_0(P)$ and $\F_1(P)$ respectively.
%Still, they are isomorphic via some isomorphism
%%
%\begin{equation}
%\label{eq:vertex_map}
%\psi:V\ni i\mapsto v_i\in\F_0(P),
%\end{equation}
%
%Then $v_1,...,v_n\in\F_0(P)$ is an enumeration of the vertices of $P$.

% $G_P=(\F_0(P),\F_1(P))$ its edge-graph with vertex set $V=\{1,...,n\}$, and for all $i\in V$, $v_i\in\F_0(P)$ the associated vertex of $P$.

\subsection{Balanced polytopes}

\begin{definition}
The polytope $P$ is called \emph{$\theta$-balanced} (or just \emph{balanced}) for some~real number $\theta\in\RR$, if
\begin{equation}
\label{eq:balanced}
\sum_{\mathclap{j\in N(i)}} v_j = \theta v_i,\quad\text{for all $i\in V$},
\end{equation}
where $N(i):=\{j\in V\mid ij\in E\}$ denotes the \emph{neighborhood} of a vertex $i\in V$.
\end{definition}

One way to interpret the balancing condition \eqref{eq:balanced} is as a kind of self-stress~con\-dition on the skeleton of $P$ (the term \enquote{balanced} is motivated from this).
For each edge $ij\in E$, the vector $v_j-v_i$ is parallel to the edge $\conv\{v_i,v_j\}$.
If $P$ is $\theta$-balanced, at each vertex $i\in V$ we have the equation
$$\sum_{\mathclap{j\in N(i)}} (v_j-v_i) = \sum_{\mathclap{j\in N(i)}} v_j - \deg(i) v_i = \big(\theta-\deg(i)\big)v_i.$$
This equation can be interpreted as two forces that cancel each other out: on the left, a contracting force along each edge (proportion only to the length of that edge), and on the right, a force repelling each vertex away from the origin (proportional to the distance of that vertex from the origin, and proportional to $\theta-\deg(i)$).

%In rigidity theory, one can interpret this line as follows: if each edge is contracting with a force proportional to its length, and each vertex is repelled from the origin with a force propotional to $\theta-\deg(i)$ and its distance from that origin, then $P$ is balanced.

A second interpretation of \eqref{eq:balanced} is via spectral graph theory.
Define the matrix
\begin{equation}
\label{eq:arrangement_matrix}
\Psi :=
\begin{pmatrix}
	\;\horzbar\!\!\!\! & v_1\T & \!\!\!\!\horzbar\;\; \\
	& \vdots & \\[0.4ex]
	\;\horzbar\!\!\!\! & v_n\T & \!\!\!\!\horzbar\;\;
\end{pmatrix}
\end{equation}
in which the $v_i$ are the rows. 
This matrix will be called the \emph{arrangement matrix} of $P$.
Note that the skeleton map $\psi$ assignes $i\in V$ to the $i$-th row of $\Psi$.
%The definition of $\Psi$ and $\psi$, and their connection, suggests a similarity to the matrix $\Phi$ and the map $\phi$ from \cref{def:eigenpolytope}. We shall come back to that later.
%
Since we use that $P\subset\RR^d$ is full-dimensional,~we have $\rank \Psi=d$.

\begin{observation}\label{res:eigenvalue}
Suppose that $P$ is $\theta$-balanced. 
The defining equation \eqref{eq:balanced} can be equivalently written as the matrix equation $A\Psi=\theta \Psi$.
In this form,~it~is~apparent that $\theta$ is an eigenvalue of the adjacency matrix $A$, and the columns of $\Psi$ are $\theta$-eigenvectors, or $\Span\Psi\subseteq\Eig_{G_P}(\theta)$.
\end{observation}

We have seen that for a balanced polytope, the columns of $\Psi$ must be eigenvectors.
But they are not necessarily a complete set of~$\theta$-eigen\-vectors, \ie\ they not necessarily span the whole eigenspace.

\begin{example}
\label{ex:neighborly}
Every centered neighborly polytope $P$ is balanced, but except if it is a simplex, it is not spectral (the latter was shown in \cref{ex:neighborly_1}).
Centered~means that
$$\sum_{i\in V} v_i = 0.$$
Since $P$ is neighborly, we have $G_P=K_n$ and  $N(i)=V\setminus\{i\}$ for all $i\in V$. Therefore
$$\sum_{\mathclap{j\in N(i)}} v_j = \sum_{\mathclap{j\in V}} v_j - v_i = -v_i,\quad\text{for all $i\in V$}.$$
And indeed, $K_n$ has spectrum $\{(-1)^{n-1},(n-1)^1\}$.
So $P$ is $(-1)$-balanced.
\end{example}

The last example shows that every neighborly polytopes can be made balanced by merely translating it.
More generally, many polytopes have a realization (of~their combinatorial type) that is balanced.
But other polytopes do not:

\begin{example}
\label{ex:prism}
Let $P\subset\RR^3$ be a triangular prism.

The spectrum of the edge-graph of $P$ is $\{(-2)^2,0^2,1^1,3^1\}$.
Note that there~is~no eigenvalue of multiplicity greater than~two.
In particular, we cannot choose three linearly independent eigenvectors to a common eigenvalue.
But if $P$ were balanced, then \cref{res:eigenvalue} tells us that the columns of the arrangement matrix $\Psi$ would be three eigenvectors to the same eigenvalue (linearly independent, since $\rank \Psi=3$), which is not possible.
And so, no realization of $P$ can be balanced.
\end{example}

\subsection{Spectral graphs and polytopes}

In the extreme case, when the columns of $\Psi$ span the whole eigenspace, we can finally give a compact definition of what we want~to~consider as \emph{spectral}:

\begin{definition}
\label{def:spectral}\quad
%Let $P\subset\RR^d$ be a polytope, and $G$ be a graph.
%
\begin{myenumerate}
	\item A polytope $P$ is called \emph{$\theta$-spectral} (or just \emph{spectral}), if its arrangement matrix $\Psi$ satisfies $\Span \Psi=\Eig_{G_P}(\theta)$.
	\item A graph is said to be \emph{$\theta$-spectral} (or just \emph{spectral}) if it is (isomorphic to) the edge-graph of~a $\theta$-spectral polytope.
\end{myenumerate}
\end{definition}

This definition is now perfectly compatible with our initial motivation for the~term \enquote{spectral} in \cref{sec:def_eigenpolytope}.
%We make some effort to connect the old and new perspective:

%The following corollary shows that everything we have considered $\theta$-spectral so far is collected under this definition.
%On the other hand, the 5-cycle is not $\theta_3$-spectral.
%A $\theta$-eigenpolytope $G$ satisfies $\Span M=\Eig_G(\theta)$ by definition, and so, if $G=G_P$ we have

%\begin{lemma}
%$P$ is $\theta$-spectral if and only if both of the following hold:
%%
%\begin{myenumerate}
%	\item $P$ is linearly equivalent to the $\theta$-eigenpolytope of its edge-graph, and
%	\item this linear map $T\in\GL(\RR^d)$ from $(ii)$ satisfies $\Psi=\Phi T$.
%\end{myenumerate}
%\end{lemma}

\begin{lemma}\quad
\label{res:naive}
\begin{myenumerate}
	\item If a polytope $P$ is $\theta$-spectral, then $P$ is linearly equivalent to the $\theta$-eigenpoly\-tope of its edge-graph (see also \cref{res:naive_polytope}).
	\item If a graph $G$ is $\theta$-spectral, then $G$ is (isomorphic to) the edge-graph of its~$\theta$-eigen\-polytope (see also \cref{res:naive_graph}).
\end{myenumerate}
\end{lemma}

In both cases, the converse is \emph{not} true.
This is intentional, to avoid the problems mentioned in \cref{ex:pentagon}.
Both statement will be proven below by formulating a more technical condition that is then actually equivalent to being spectral.

%\begin{proposition}
%\label{res:naive_polytope}
%Let $P$ be a polytope with arrangement matrix $\Psi$, $G_P$ its edge-graph, $P_{G_P}(\theta)$ the $\theta$-eigenpolytope of its edge-graph, and $\Phi$ the matrix \eqref{eq:eigenpolytope_matrix}~defined for $P_{G_P}(\theta)$ in \cref{def:eigenpolytope}.
%
%Then, the polytope $P$ is $\theta$-spectral {if and only if} $P$ and $P_{G_P}(\theta)$ are linearly equivalent via a linear map $T\in\GL(\RR^d)$ that satisfies $\Phi=\Psi T$.
%\end{proposition}

%This also proves \cref{res:naive} $(i)$.

%
%\begin{proposition}
%\label{res:naive_graph}
%Let $G$ be a graph, $P_G(\theta)=:Q$ its $\theta$-eigenpolytope, $G_Q$ the edge-graph thereof, and $\Phi$ the matrix \eqref{eq:eigenpolytope_matrix} defined for $P_{G}(\theta)$ in \cref{def:eigenpolytope}.
%
%Then, the graph $G$ is $\theta$-spectral if and only if $G$ is isomorphic to the edge-graph of its $\theta$-eigenpolytope via some isomorphism $\sigma:V(G)\to V(G_Q)$ that satisfies $\Phi=\Pi_\sigma \Psi$, where $\Psi$ is the arrangement matrix of some $\theta$-spectral polytope $P$ with edge-graph $G$ .
%\end{proposition}

\begin{proposition}
\label{res:naive_polytope}
A polytope $P$ is $\theta$-spectral if and only if it is linearly equivalent to the $\theta$-eigenpolytope of its edge-graph via some linear map $T\in\GL(\RR^d)$ for which the following diagram commutes:
\begin{equation}
\label{eq:diagram_polytope}
\begin{tikzcd}
P \arrow[r, "T"] & P_{G_P}(\theta)                           \\
                 & G_P \arrow[lu, "\psi"] \arrow[u, "\phi"']
\end{tikzcd}
\end{equation}
where $\phi$ and $\psi$ denote the eigenpolytope map and skeleton map respectively.
\end{proposition}

%Recall the definitions of $\phi$ and $\psi$.
%The map $\phi:V(G)\to\RR^d\supseteq P_G(\theta)$ from~\eqref{eq:eigenpolytope_map} is defined whenever we compute an eigenpolytope of a graph $G$.
%The map $\psi:V(G_P)$ $\to \F_0(P)$ from \eqref{eq:vertex_map} is defined whenever we have a polytope $P$ and its edge-graph $G_P$.
%These maps enumerate the rows of the matrices $\Phi$ and $\Psi$ respectively.

\begin{proof}%[Proof of \cref{res:naive_polytope}]
By definition, the $\theta$-eigenpolytope of $G_P$ satisfies $\Span \Phi=\Eig_{G_P}(\theta)$, where $\Phi$ is the corresponding eigenpolytope matrix.

Now, by definition, $P$ is $\theta$-spectral if and only if $\Span \Psi = \Eig_{G_P}(\theta)$, where $\Psi$ is its arrangement matrix.
But by \cref{res:same_column_span}, $\Phi$ and $\Psi$ have the same span if and only of their rows are related by some invertible linear map $T\in\GL(\RR^d)$, that is, $\Psi T=\Phi$, or $T\circ \psi=\phi$. The latter expresses exactly that \eqref{eq:diagram_polytope} commutes.
\end{proof}

This also proves \cref{res:naive} $(i)$.
%Proving the equivalent statement for \cref{res:naive} $(ii)$ is more tedious.

\begin{proposition}
\label{res:naive_graph}
A graph $G$ is $\theta$-spectral if and only if the eigenpolytope map~$\phi\:$ $V(G)\to \RR^d$ provides an isomorphism between $G$ and the skeleton of its $\theta$-eigenpoly\-tope $P_G(\theta)$.
\begin{proof}
Suppose first that $G$ is $\theta$-spectral.
Then there is a $\theta$-spectral polytope $Q$~with edge-graph~$G_Q=G$ and skeleton map $\psi\:V(G_Q)\to\F_0(Q)$.
By \cref{res:naive} $(i)$, $Q$ is linearly equivalent to $P_G(\theta)$ via some linear map $T\in\GL(\RR^d)$.
By \cref{res:naive_polytope}, the eigenpolytope map satisfies $\phi=T\circ \psi$.
Since $T$ induces an isomorphism between the skeleta of $Q$ and $P_G(\theta)$, and $\psi$ is an isomorphism between $G$ and the skeleton of $Q$, we find that $\phi$ must be an isomorphism between $G$ and the skeleton of $P_G(\theta)$.
This shows one direction.

For the converse, suppose that $\phi$ is an isomorphism.
Set $P:=P_G(\theta)$ and let~$G_P$ be its edge-graph with skeleton map $\psi\:V(G_P)\to \F_0(P)$.
Then $\sigma:=\psi^{-1}\circ\phi$ is a graph isomorphism between $G$ and $G_P$.
So, since $G\cong G_P$, each eigenpolytope of $G$ is also an eigenpolytope of $G_P$.
We can therefore choose $P_{G_P}(\theta)=P_G(\theta)$, with corresponding eigenpolytope map $\phi':=\sigma^{-1}\circ\phi$.
In sum, the outer square in the following diagram commutes:
\begin{center}
\begin{tikzcd}
G \arrow[r, "\sigma"] \arrow[d, "\phi"'] & G_P \arrow[d, "\phi'"] \arrow[ld, "\psi"'] \\
\mathllap{P:=\,}P_G(\theta) \arrow[r, "\Id"']    & P_{G_P}(\theta)                           
\end{tikzcd}
\end{center}
Also, by construction of $\sigma$, the upper triangle commutes.
In conclusion, the lower triangle must commute as well, which is exactly \eqref{eq:diagram_polytope} with $T=\Id$. This proves that $P$~is $\theta$-spectral via \cref{res:naive_polytope}.
Since $G$ is isomorphic to $G_P$, $G$ is $\theta$-spectral.
\end{proof}
\end{proposition}

\noindent
This also proves \cref{res:naive} $(ii)$.

It is also possible to give a definition of spectral graphs purely in terms of graph theory, without any explicit reference to polytopes:

\begin{lemma}
\label{res:spectral_2}
A graph $G$ is $\theta$-spectral if and only if it satisfies both of the following:
\begin{myenumerate}
	\item for each vertex $i\in V$ exists a $\theta$-eigenvector $u=(u_1,...,u_n)\in\Eig_G(\theta)$~whose single largest component is $u_i$, or equivalently,
$$\Argmax_{k\in V} u_k = \{i\}.$$
	\item any two vertices $i,j\in V$ form an edge $ij\in E$ in $G$ if and only~if there is a $\theta$-eigenvector $u=(u_1,...,u_n)\in\Eig_G(\theta)$ whose only two largest components are $u_i$ and $u_j$, or equivalently,
$$\Argmax_{k\in V} u_k = \{i,j\}.$$
\end{myenumerate}
\end{lemma}

This characterization of spectral graphs can be interpreted as follows: a spectral graph can be reconstructed from knowing a single eigenspace, rather than, say, all eigenspaces and their associated eigenvalues.

\begin{proof}[Proof of \cref{res:spectral_2}]

Let $P_G(\theta)\subset\RR^d$ be the $\theta$-eigenpolytope of $G$ with eigenpolytope matrix $\Phi$ and eigenpolytope map $\phi\:V\ni i\mapsto v_i\in\RR^d$.

Since $\Span\Phi=\Eig_G(\theta)$, the eigenvectors $u=(u_1,...,u_n)\in\Eig_G(\theta)$ are exactly the vectors that can be written as $u=\Phi x$ for some $x\in\RR^d$.
If then $e_k\in\RR^n$ denotes the $k$-th standard basis vector, we have
$$u_k = \<u,e_k\> = \<\Phi x, e_k\> = \<x,\Phi\T\! e_k\> = \<x, v_k\>.$$
Therefore, there is a $\theta$-eigenvector $u=(u_1,...,u_n)\in\Eig_G(\theta)$ with
$\Argmax_{k\in V} u_k = \{i_1,...,i_m\}$
if and only if there is a vector $x\in\RR^d$ with
$$\Argmax_{k\in V} \<x,v_k\> = \{i_1,...,i_m\}.$$
But this last line is exactly what it means for $\conv\{v_{i_1},...,v_{i_m}\}$ to be a face of $P_G(\theta)$ $=\conv\{v_1,...,v_n\}$ (and $x$ is a normal vector of that face).

In this light, we can interpret $(i)$ as stating that $v_1,...,v_n$ form $n$ distinct vertices of $P_G(\theta)$, and $(ii)$ as stating that $\conv\{v_i,v_j\}$ is an edge of $P_G(\theta)$ if and only if $ij\in E$.
And this means exactly that $\phi$ is a graph isomorphism between $G$ and the skeleton of $P_G(\theta)$.
By \cref{res:naive_graph}, this is equivalent to $G$ being $\theta$-spectral.
%
%
%
%
%
%
%
%The set of vertices of the $\theta$-eigenpolytope $P:=P_G(\theta)\subset\RR^d$ is a~sub\-set of $\{v_1,...,v_n\}$, where the $v_i$ are as defined in \cref{def:eigenpolytope}.
%
%By the usual definition of \enquote{face of a polytope}, a set of points $\{v_{i_1},...,v_{i_m}\}\subseteq P$ is now the vertex set of a face of the $\theta$-eigenpolytope if and only of there is a vector~$x\in$ $\RR^d$ (a normal vector) with
%%
%$$\Argmax_{k\in V} \<x,v_k\> = \{i_1,...,i_m\}.$$
%%
%If $\Phi$ is the matrix \eqref{eq:eigenpolytope_matrix}, $u=(u_1,...,u_n):=\Phi x$, and $e_k\in\RR^n$ the $k$-th canonical~basis vector, then
%%
%$$\<x,v_k\> = \<x,\Phi\T\! e_k\> = \<\Phi x,e_k\> = \<u,e_k\> = u_k.$$
%%
%And since $x\mapsto u:=\Phi x$ defines an isomorphism between $\RR^d$ and $\Span \Phi=\Eig_G(\theta)$, we found that $\{v_{i_1},...,v_{i_m}\}\subseteq P$ is the vertex set of a face of $P$ if and only if there~is a~$\theta$-eigenvector $u\in\Eig_G(\theta)$ with
%%
%$$\Argmax_{k\in V} u_k = \{i_1,...,i_m\}.$$
%
%By the previous considerations, $(i)$ is equivalent to: the $v_1,...,v_n$ form distinct vertices of $P$; and $(ii)$ is equivalent~to: $\conv\{v_i,v_j\}$ is an edge of $P$ if and only if $ij\in E$.
%In conclusion, $(i)$ and $(ii)$ in combination are equivalent to: the map $\phi\:i\mapsto v_i$ is an isomorphism between $G$ and the skeleton of $P_G(\theta)$.
%By \cref{res:naive_graph}, this is equivalent to $G$ being $\theta$-spectral.
%In still other words, $(i)$ and $(ii)$ combined are equivalent to $G$ being isomorphic to the skeleton of $P$.
\end{proof}

In practice, to reconstruct a spectral graph from an eigenspace, the steps could be the following: given a subspace $U\subseteq\RR^n$ (the claimed eigenspace), then
\begin{myenumerate}
	\item choose any basis $u_1,...,u_d\in\RR^n$ of $U$,
	\item build the matrix $\Phi=(u_1,...,u_d)\in\RR^{n\x d}$ in which the $u_i$ are the columns,
	\item define $v_i$ as the $i$-th \emph{row} of $\Phi$,
	\item define $P:=\conv\{v_1,...,v_n\}\subset\RR^d$ as the convex hull of the $v_i$,
	\item the reconstructed graph $G=G_P$ is then the edge-graph of $P$.
\end{myenumerate}

\subsection{Properties of spectral polytopes}

We discuss two properties of spectral polytopes that make them especially interesting in polytope theory. 

\subsubsection*{Reconstruction from the edge-graph}
\label{sec:spectral_reconstruction}
The edge-graph of a general polytope carries little information about that polytope \ie\ given only its edge-graph, we can often not reconstruct the polytope from this (up to combinatorial equivalence).
Often, one cannot even deduce the dimension of the polytope from its edge-graph.
Reconstruction might be possible in certain special cases, as \eg\ for 3-dimensional polyhedra, simple polytopes or zonotopes.
The spectral polytopes provide another such class.

\begin{theorem}\label{res:reconstruction}
A $\theta_k$-spectral polytope is uniquely~determined by its edge-graph up to invertible linear transformations.
\end{theorem}

The proof is simple: 
every $\theta_k$-spectral polytope is linearly equivalent to the~$\theta_k$-eigenpolytope of its edge-graph (by \cref{res:naive} $(i)$).
Our definition of the \mbox{$\theta_k$-eigen}\-polytope already suggests an explicit procedure to construct it (a script for this~is included in \cref{sec:appendix_mathematica}).
This property of spectral polytopes appears more exciting when applied to graph classes that are not obviously spectral (see \cref{sec:edge_transitive}).

\subsubsection*{Realizing symmetries of the edge-graph}
\label{sec:spectral_symmetries}
Every Euclidean symmetry of a polytope induces a combinatorial symmetry on its edge-graph.
The converse is far from true.
Think, for example, about a rectangle that is not a square.
Even worse, it can happen that a polytope does not even have a realization that realizes all the symmetries of its edge-graph (\eg\ the polytope constructed in \cite{bokowski1984combinatorial}).
%Again, spectral polytopes are special in this regard.

We have previously discussed (in \cref{res:realizing_symmetries}) the existence of a homomorphism $\Aut(G)\to\Aut(P_G(\theta))$ between the symmetries of a graph $G$ and the symmetries of its eigenpolytopes.
There are two caveats:
\begin{myenumerate}
	\item this is not necessarily an isomorphism, and
	\item it says nothing about the symmetries of the edge-graph of $P_G(\theta)$, as this one needs not to be isomorphic to $G$
\end{myenumerate}
Still, it suffices to makes statement of the following form:
if $G$ is vertex-transitive, then so are all its eigenpolytopes.
This might not work with other transitivities, as for example edge-transitivity.

%We have previously discussed (in \cref{res:realizing_symmetries}) that an eigenpolytope of a graph $G$ \enquote{realizes} all the symmetries of $G$.
%This was meant in weak sense: it can happen that non-trivial symmetries of $G$ become trivial symmetries of the eigenpolytope.
%Still, if, for example, $G$ is vertex-transitive, so are all its eigenpolytopes.
%Other symmetries do not translate as well, \eg\ edge-transitivity.

This is no concern for spectral graphs/polytopes:

\begin{theorem} \quad
\label{res:symmetries}
\begin{myenumerate}
	\item If $G$ is $\theta$-spectral, then $P_G(\theta)$ realizes all its symmetries, which includes
	$$\Aut(G)\cong\Aut(P_G(\theta))$$
	 via the map $\sigma\mapsto T_\sigma$ given in \cref{res:realizing_symmetries}, as wells as that $T_\sigma$ permutes the vertices and edges of $P_G(\theta)$ exactly as $\sigma$ permutes the vertices and edges of the graph $G$.
	\item If $P$ is $\theta$-spectral, then $P$ has a realization that realizes all the symmetries~of its edge-graph, namely, the $\theta$-eigenpolytope of its edge-graph.
\end{myenumerate}
	%
%\begin{proof}
%If $G$ is spectral, then the vertices of $P_G(\theta)$ are exactly the $n$ distinct points $v_1,...,v_n$.
%As stated in \cref{res:realizing_symmetries}, $T_\sigma$ permutes the $v_i$ as prescribed by $\sigma$, which now applies to the vertices, and by this, also to the edges.
%A non-trivial symmetry $\sigma\in\Aut(G)$ does not fixes all vertices of $G$, and hence, neither can $T_\sigma$.
%Therefore, the homomorphism $\sigma\mapsto T_\sigma$ is an isomorphism.
%This proves $(i)$.
%
%
%
%\end{proof}

%
%If $G$ is $\theta$-spectral, then $\Aut(G)\cong \Aut(P_G(\theta))$. % via
%%%
%%$$\Aut(G)\ni \sigma \quad \mapsto\quad  \phi\circ \sigma\circ \phi^{-1}\in\Aut(P_G(\theta)),$$
%%%
%%where $\phi$ is the corresponding eigenpolytope map.
%
%In particular, every $\theta$-spectral polytope has a realization that realizes all the symmetries of its edge-graph, namely, the $\theta$-eigenpolytope of its edge-graph.
%
%\begin{proof}
%A homomorphism ... \msays{Here we need that $G$ is isomorphic to the edge graph of $P_G(\theta)$.}
%
%By \cref{res:naive_polytope}, a $\theta$-spectral polytope $P$ is linearly equivalent to $P_{G_P}(\theta)$, and so the latter is a realization of $P$.
%Since $G_P$ is $\theta$-spectral, $P_{G_P}(\theta)$ also realizes all the symmetries of $G_P$ (as proven above).
%\end{proof}
\end{theorem}

This is mostly straight forward, with large parts already addressed in \mbox{\cref{res:realizing_symmetries}}.
The major difference is that for spectral graphs $G$ the eigenpolytope has exactly the distinct vertices $v_1,...,v_n\in\RR^d$.
The statement from \cref{res:realizing_symmetries} that $T_\sigma$~per\-mutes the $v_i$ as prescribed by $\sigma$, then becomes, that $T_\sigma$ permutes the \emph{vertices} as prescribed by $\sigma$, and hence also the edges.
Also, since the $v_i$ are distinct, no non-trivial symmetry $\sigma$ can result in trivial $T_\sigma$, making $\sigma\mapsto T_\sigma$ into a group \emph{isomorphism}.

%This is also seen easily: by \cref{res:realizing_symmetries}, each symmetry of the graph corresponds to a symmetry of the polytope, and every symmetry of the polytope must restrict to a symmetry of the graph.
%In contrast to general eigenpolytope, this time no two vertices of $G$ get mapped to the same point, and so no non-trivial symmetry of $G$ can be mapped to the identity.
%
%\cref{res:realizing_symmetries} gave the explicit map
%%
%$$\Aut(G)\ni\sigma\quad\mapsto \quad T_\sigma:= \Phi\T\Pi_\sigma\Phi\in\Aut(P_G(\theta)).$$
%%
%This can be written as $\sigma\mapsto \phi\circ\sigma\circ\phi^{-1}$, where $\phi$ is the eigenpolytope map.

For part $(ii)$ merely recall that the eigenpolytope $P_{G_P}(\theta)$ is indeed a realization of $P$ by \cref{res:naive} $(i)$.

The major consequence of this is, that for spectral graphs/polytopes also more complicates types of symmetries translate between a polytope and its graph, as \eg\ edge-transitivity (see also \cref{sec:edge_transitive}).

\section{The Theorem of Izmestiev}
\label{sec:izmestiev}

We introduce our, as of yet, most powerful tool for proving that certain polytopes are $\theta_2$-spectral.
For this, we make use of a more general theorem by Izmestiev \cite{izmestiev2010colin}, first proven in the context of the Colin de Verdière graph invariant.
The proof of this theorem requires techniques from convex geometry, most notably, mixed volumes, which we not address here.
We need to introduce some terminology.

%It can be interesting to study the properties that follow from a polytope being balanced or spectral.
%But it is equally interesting and important to actually determine what polytopes are spectral, or can be made spectral while preserving the combinatorial type.
%
%Of course, for each particular polytope it is an easy task to check whether it is balanced/spectral.
%Much harder it is to find other properties of polytopes that imply it being spectral.
%One tool that can be used for this is the following theorem of Izmestiev, first proved in the context of the Colin de Verdiere graph invariant:

%Large groups of polytopes can be recognized as $\theta_2$-spectral with the help of a result of Izmestiev.
%There is a remarkable construction due to Izmestiev \cite{izmestiev2010colin}.

As before, let $P\subset\RR^d$ denote a full-dimensional polytope of dimension $d\ge 2$, with~edge-graph $G_P=(V,E),V=\{1,...,n\}$ and vertices $v_i\in \F_0(P),i\in V$.
Recall, that the \emph{polar dual} of $P$ is the polytope
$$P^\circ:=\{x\in\RR^d\mid \<x,v_i\>\le 1\text{ for all $i\in V$}\}.$$
We can replace the $1$-s in this definition by variables $c=(c_1,...,c_n)$, to obtain% the following definition:
$$P^\circ(c):=\{x\in\RR^d\mid\<x,v_i\>\le c_i\text{ for all $i\in V$}\}.$$
The usual polar dual is then $P^\circ=P^\circ(1,...,1)$.

\begin{figure}[h!]
\includegraphics[width=0.85\textwidth]{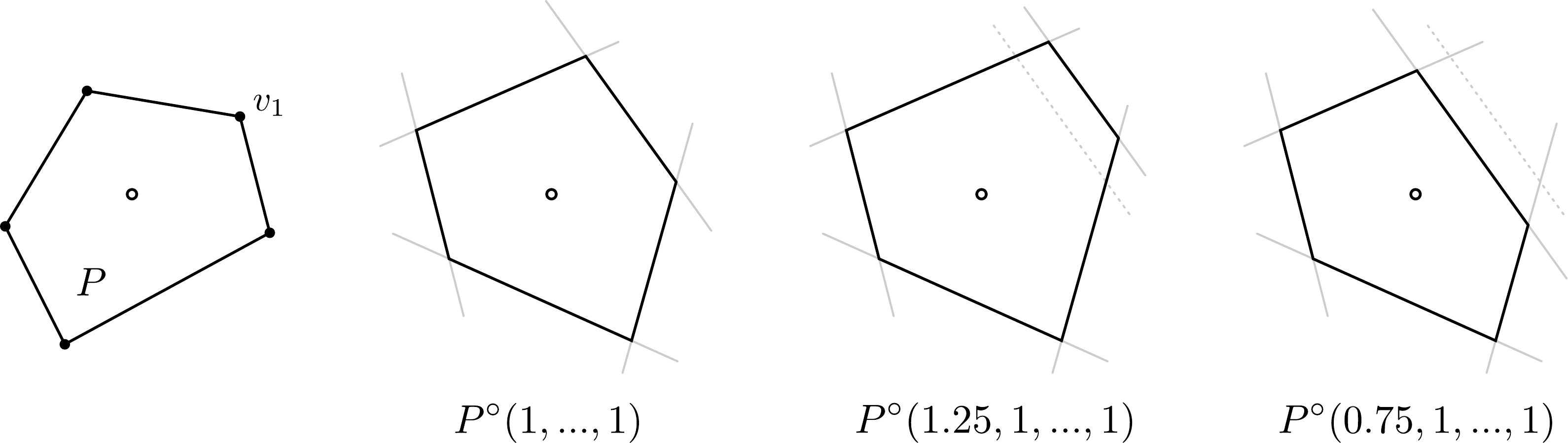}
\caption{Visualization of $P^\circ(c)$ for different values of $c\in\RR^n$.}
\end{figure}

In the following, $\vol(\free)$ denotes the volume of convex sets in $\RR^d$ (\wrt\ the usual Lebesgue measure).
Note that the function $\vol(P^\circ(c))$ is differentiable in $c$, and so we can compute partial derivatives \wrt\ the components of $c$.

\begin{theorem}[Izmestiev \cite{izmestiev2010colin}, Theorem 2.4]\label{res:izmestiev}
%Let $P\subset\RR^d$ be a polytope, $G_P$ its~edge-graph and $v\:V\to\RR^d$ its skeleton. 
%For $c=(c_1,...,c_n)\in\RR^n$ define
%%
%$$P^\circ(c):=\{x\in\RR^d\mid\<x,v_i\>\le c_i\text{ for all $i\in V$}\}.$$
%%
%Note that $P^\circ=P^\circ(1,...,1)$ is the polar dual of $P$.
Define a matrix $X\in\RR^{n\x n}$ with~compo\-nents
$$X_{ij}:=-\frac{\partial^2 \vol(P^\circ(c))}{\partial c_i\partial c_j}\Big|_{c=(1,...,1)}.$$
The matrix $X$ has the following properties:
\begin{myenumerate}
	%\item whenever $ij\in E$, and $\sigma_i,\sigma_j\in\F_{d-1}(P^\circ)$ are the dual faces to $v_i$ \resp\ $v_j$ in the dual polytope $P^\circ=P^\circ(1,...,1)$, 
	%then
	%
	%$$X_{ij}=\frac{\vol(\sigma_i\cap\sigma_j)}{\|v_i\|\|v_j\|\sin\angle(v_i,v_j)},$$
	
	\item $X_{ij}< 0$ whenever $ij\in E(G_P)$,
	\item $X_{ij}=0$  whenever $ij\not\in E(G_P)$,
	%\item $X_{ii}> 0$ for all $i\in V(G_P)$,
	\item $X\Psi=0$ (where $\Psi$ is the arrangement matrix of $P$),
	\item $X$ has a unique negative eigenvalue, and this eigenvalue is simple,
	\item $\dim\ker X=d$.
\end{myenumerate}
\end{theorem}

One can view the matrix $X$ as some kind of adjacency matrix of~a vertex- and edge-weighted version of $G_P$.
Part $(iii)$ states that $v$ satisfies a weighted form of the balancing condition \eqref{eq:balanced} with eigenvalue zero.
Since $\rank \Psi=d$, part $(v)$ states that $\Span \Psi$ is already the whole 0-eigenspace.
And part $(iv)$ states that zero is the second smallest eigenvalue of $X$.

\begin{theorem}\label{res:implies_spectral}
Let $X\in\RR^{n\x n}$ be the matrix defined in \cref{res:izmestiev}. 
If we have
\begin{myenumerate}
	\item $X_{ii}$ is independent of $i\in V(G_P)$, and
	\item $X_{ij}$ is independent of $ij\in E(G_P)$,
\end{myenumerate}
then $P$ is $\theta_2$-spectral.
\begin{proof}
By assumption there are $\alpha,\beta\in\RR$, $\beta>0$, so that $X_{ii}=\alpha$ for all vertices~$i\in$ $V(G_P)$, and $X_{ij}=\beta<0$ for all edges $ij\in E(G_P)$ (we have $\beta<0$ by \cref{res:izmestiev} $(i)$).
We can write this as
$$X=\alpha \Id + \beta A\quad\implies\quad (*)\;A=\frac\alpha\beta \Id+\frac1\beta X,$$
where $A$ is the adjacency matrix of $G_P$.
By \cref{res:izmestiev} $(iv)$ and $(v)$, the matrix $X$ has second smallest eigenvalue zero of multiplicity $d$.
By \cref{res:izmestiev}  $(iii)$, the columns of $M$ are the corresponding eigenvectors.
Since $\rank \Psi=d$ we find that these are all the eigenvectors and $\Span \Psi$ is the 0-eigenspace of $X$.

By $(*)$ the eigenvalues of $A$ are the eigenvalues of $X$, but scaled by $1/\beta$ and shifted by $\alpha/\beta$. Since $1/\beta <0$, the second-\emph{smallest} eigenvalue of $X$ gets mapped onto the second-\emph{largest} eigenvalue of $A$.
Therefore, $A$ (and also $G_P$) has second-largest eigenvalue $\theta_2=\alpha/\beta$ of multiplicity $d$, and $\Span \Psi$ is the corresponding eigenspace.
By definition, $P$ is then the $\theta_2$-eigenpolytope of $G_P$ and is therefore~$\theta_2$-spectral.
\end{proof}
\end{theorem}

It is unclear whether \cref{res:implies_spectral} already characterizes $\theta_2$-spectral polytopes, or even spectral polytopes in general (see also \cref{q:characterization}).

%\begin{corollary}
%Let $P\subset\RR^d$ be a polytope, so that
%%
%\begin{myenumerate}
%	\item its edge-graph is regular,
%	\item it is inscribed,
%	\item all edges are of the same length,
%	\item the ridges of $P^\circ$ have the same volume.
%\end{myenumerate}
%%
%Then $P$ is $\theta_2$-spectral.
%\end{corollary}

\tempnewpage
\section{Edge-transitive polytopes}
\label{sec:edge_transitive}

We apply \cref{res:implies_spectral} to edge-transitive polytopes, that is, to polytopes for~which the Euclidean symmetry group $\Aut(P)\subset\Ortho(\RR^d)$ acts transitively on the edge set $\F_1(P)$.
No classification of edge-transitive polytopes is known.
Some \mbox{edge-transitive} polytopes are listed in \cref{sec:classification}.

Despite the name of this section, we are actually going to address polytopes that are simultaneously vertex- and edge-transitive.
This is not a huge deviation from the title: as shown in \cite{winter2020polytopes}, edge-transitive polytopes in dimension $d\ge 4$ are always also vertex-transitive, and the exceptions in lower dimensions are few (a continuous family of $2n$-gons for each $n\ge 2$, and two exceptional polyhedra).
%The exceptions in lower dimensions are manageable: for $d=3$, there are only two 

%In this section we discuss the class of polytopes that are simultaenously vertex- and edge-transitive.
%This is clearly stronger than being just vertex-transitive.
%However, as shown in \cite{winter2020polytopes}, every edge-transitive polytope in dimendion $d\ge 4$ is also vertex-transitive, and so this distinction only matters for dimension up to three (in which there are only two exceptions, besides polygons).

\cref{res:implies_spectral} can be directly applied to simultaneously vertex- and edge-transitive polytopes, and so we have

\begin{corollary}
\label{res:vertex_edge_transitive_cor}
A simultaneously vertex- and edge-transitive polytope is $\theta_2$-spectral.
\end{corollary}

We collect all the notable consequences in the following theorem:

\begin{theorem}\label{res:edge_vertex_transitive}
If $P\subset\RR^d$ is simultaneously vertex- and edge-transitive, then
\begin{myenumerate}
\item $\Aut(P)\subset\RR^d$ is irreducible as a matrix group.
	%\item $\Aut(P)$ is irreducible,
	%\item $P$ is $\theta_2$-spectral, \ie\ $P$ is the $\theta_2$-eigenpolytope of its edge-graph,
%	\item the combinatorial type of $P$ is uniquely determined by its edge-graph, or~more precisely, $P$ is (up to invertible linear transformations) the unique simultaneously vertex- and edge-transitive polytope with~edge-graph $G_P$.
\item $P$ is uniquely determined by its edge-graph up to scale and orientation.\footnote{This shows that $P$ is \emph{perfect}, \ie\ is the unique maximally symmetric realization of its combinatorial type. See \cite{gevay2002perfect} for an introduction to perfect polytopes.}
%\item $P$ is the $\theta_2$-eigenpolytope of its its edge-graph 
%(among simultaneously vertex- and edge-transitive polytopes) 
	\item $P$ realizes all the symmetries of its edge-graph. 
%	\item $P$ is uniquely determined by its edge-graph, in the sense, that it is~the~unique simultaneously vertex- and edge-transitive polytope (up to scale and orientation) with edge-graph $G_P$.
%	\item $P$ is the unique simultaneously vertex- and edge edge-transitive realization of its combinatorial type (up to scale and orientation)\,\footnote{A polytope which is the unique most symmetric realization of its combinatorial type is also known as a \emph{perfect} polytope, see \eg\  \cite{gevay2002perfect}.}\!.
	\item if $P$ has edge length $\ell$ and circumradius $r$, then
	\begin{equation}
	\label{eq:circumradius}
	\frac{\ell}r = \sqrt{\frac{2\lambda_2}{\deg(G_P)}} = \sqrt{2\Big(1-\frac{\theta_2}{\deg(G_P)}\Big)},
	\end{equation}
	where $\deg(G_P)$ is the vertex degree of $G_P$, and $\lambda_2=\deg(G_P)-\theta_2$ denotes its second smallest Laplacian eigenvalue.
	\item if $\alpha$ is the dihedral angle of the polar dual $P^\circ$, then
	\begin{equation}
	\label{eq:dihedral_angle}
	\cos(\alpha)=-\frac{\theta_2}{\deg(G_P)}.	
	\end{equation}
\end{myenumerate}
\begin{proof}
The complete proof of $(i)$ and $(ii)$ has to be postponed until \cref{sec:rigidity} (see \cref{res:edge_transitive_rigid}).
Concerning $(ii)$, from \cref{res:vertex_edge_transitive_cor} and \cref{res:reconstruction} already follows that $P$ is determined by its edge-graph up to \emph{invertible linear transformations}, but not necessarily only up to scale and orientation.

Part $(iii)$ follows from \cref{res:symmetries}.
Part $(iv)$ and $(v)$ were proven (in a more general setting) in \cite[Proposition 4.3]{winter2020symmetric}.
This applies literally to $(iv)$.
For $(v)$, note the following: if $\sigma_i\in\F_{d-1}(P^\circ)$ is the facet of the polar dual $P^\circ$ that corresponds to the vertex $v_i\in\F_1(P)$, then the dihedral angle between $\sigma_i$ and $\sigma_j$ is $\pi-\angle(v_i,v_j)$.
The latter expression was proven in \cite{winter2020symmetric} to agree with  \eqref{eq:dihedral_angle}.
\end{proof}
\end{theorem}

It is worth emphasizing that large parts of \cref{res:edge_vertex_transitive} do not apply to polytopes of a weaker symmetry, as \eg\ vertex-transitive polytopes.
Prisms are counterexamples to both $(i)$ and $(ii)$.
There are vertex-transitive neighborly polytopes (other than simplices) and they are counterexamples to $(ii)$ and $(iii)$.

\begin{remark}
\label{rem:edge_not_vertex}
There are two edge-transitive polyhedra that are not vertex-transitive: the \emph{rhombic dodecahedron} and the \emph{rhombic triacontahedron} (see also \cref{fig:edge_transitive}).~Only the former is $\theta_2$-spectral, and the latter is not spectral for any eigenvalue (this was already mentioned in \cite{licata1986surprising}).
Since the rhombic dodecahedron is not vertex-transitive, nothing of this follows from \cref{res:vertex_edge_transitive_cor}.
However, this polytope satisfies the conditions of \cref{res:implies_spectral}, which seems purely accidental.
It is the only known spectral polytope that is not vertex-transitive.
\end{remark}

\subsection{Rigidity and irreducibility}
\label{sec:rigidity}

The goal of this section is to prove the missing part of \cref{res:edge_vertex_transitive}:

\begin{theorem}
\label{res:edge_transitive_rigid}
If $P\subset\RR^d$ is simultaneously vertex- and edge-transitive, then
\begin{myenumerate}
	\item $\Aut(P)\subset\Ortho(\RR^d)$ is irreducible as a matrix group, and
	\item $P$ is determined by its edge-graph up to scale and orientation.
	%\item $P$ is the $\theta_2$-eigenpolytope of its edge-graph up to scale and orientation.
\end{myenumerate}
\end{theorem}

To prove \cref{res:edge_transitive_rigid}, we make use of \emph{Cauchy's rigidity theorem} for polyhedra (with its beautiful proof listed in \cite[Section 12]{aigner2010proofs}).
It states that every~polyhedron is uniquely determined by its combinatorial type and the shape of its faces.
This was generalized by Alexandrov to general dimensions $d\ge 3$ (proven \eg\ in \cite[Theorem 27.2]{pak2010lectures}):

\begin{theorem}[Alexandrov]
\label{res:alexandrov}
Let $P_1,P_2\subset\RR^d, d\ge 3$ be two polytopes, so that
\begin{myenumerate}
	\item $P_1$ and $P_2$ are combinatorially equivalent via a face lattice~isomorphism~$\phi:\F(P_1)\to\F(P_2)$, and
	\item each facet $\sigma\in\F_{d-1}(P_1)$ is congruent to the facet $\phi(\sigma)\in\F_{d-1}(P_2)$.
\end{myenumerate}
Then $P_1$ and $P_2$ are congruent, \ie\ are the same up to orientation.
\end{theorem}

\begin{proposition}
\label{res:regular_rigid}
Let $P_1,P_2\subset\RR^d$ be two combinatorially equivalent polytopes,~each of which has
\begin{myenumerate}
	\item all vertices on a common sphere (\ie\ is inscribed), and
	\item all edges of the same length $\ell_i$.
\end{myenumerate}
Then $P_1$ and $P_2$ are the same up to scale and orientation.
\begin{proof}[Proof.\!\!]\footnote{This proof was proposed by the user \emph{Fedor Petrov} on MathOverflow \cite{petrovMO}.}
W.l.o.g.\ assume that $P_1$ and $P_2$ have the same circumradius, otherwise~re\-scale $P_2$.
It then suffices to show that $P_1$ and $P_2$ are the same up to orientation.

We proceed with induction by the dimension $d$.
The induction base is given by $d=2$, which is trivial, since any two inscribed polygons with constant edge length are regular and thus completely determined (up to scale and orientation) by their number of vertices.

Suppose now that $P_1$ and $P_2$ are combinatorially equivalent polytopes of dimension $d\ge 3$ that satisfy $(i)$ and $(ii)$.
Let $\phi$ be the face lattice isomorphism between them.
Let $\sigma\in\F_{d-1}(P_1)$ be a facet of $P_1$, and $\phi(\sigma)$ the corresponding facet in $P_2$.
In particular, $\sigma$ and $\phi(\sigma)$ are combinatorially equivalent.
Furthermore, both $\sigma$ and $\phi(\sigma)$ are of dimension $d-1$ and satisfy $(i)$ and $(ii)$. This is obvious for $(ii)$, and for $(i)$ recall that facets of inscribed polytopes are also inscribed.
By induction hypothesis, $\sigma$ and $\phi(\sigma)$ are then congruent.
Since this holds for all facets $\sigma\in\F_{d-1}(P_1)$, \cref{res:alexandrov} tells us that $P_1$ and $P_2$ are congruent, that is, the same up to orientation.
\end{proof}
\end{proposition}

We can now prove the main theorem of this section:

\begin{proof}[Proof of \cref{res:edge_transitive_rigid}]
By \cref{res:edge_vertex_transitive} the combinatorial type of $P$ is determined by its edge-graph.
By vertex-transitivity, all vertices are on a sphere.
By edge-transitivity, all edges are of the same length.
We can then apply \cref{res:regular_rigid} to obtain that $P$ is unique up to scale and orientation.
This proves $(ii)$.

Suppose now, that $\Aut(P)$ is not irreducible, but that $\RR^d$ decomposes as $\RR^d=W_1\oplus W_2$ into non-trivial orthogonal $\Aut(P)$-invariant subspaces.
Let $T_\alpha\in\GL(\RR^d)$ be the linear map that acts as identity on $W_1$, but as $\alpha\Id$ on $W_2$ for some $\alpha >1$.
Then $T_\alpha P$ is a non-orthogonal linear transformation of $P$ (in particular, combinatorially equivalent), on which $\Aut(P)$ still acts vertex- and edge-transitively.
By $(ii)$, this cannot be. Hence $\Aut(P)$ must be irreducible, which proves $(i)$.
\end{proof}

\subsection{A word on classification}
\label{sec:classification}

Despite the simple appearance of the definition of an edge-transitive polytope, %(one should have already thought about them, right?), 
no classification was obtained so far.

There exists a classification of the 3-dimension edge-transitive polyhedra: besides the Platonic solids, these are the ones shown in \cref{fig:edge_transitive} (nine in total).

\begin{figure}[h!]
\includegraphics[width=0.75\textwidth]{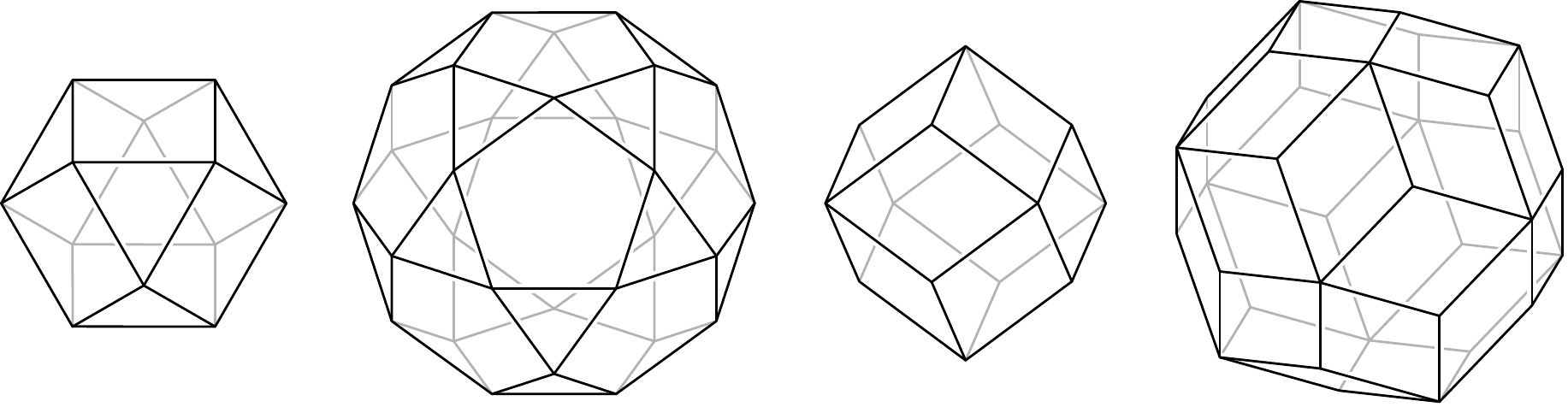}
\caption{From left to right, these are: the cuboctahedron, the icosido\-decahedron, the rhombic dodecahedron, and the rhombic triacontahedron.}
\label{fig:edge_transitive}
\end{figure}

There are many known edge-transitive polytopes in dimension $d\ge 4$ (so we~are not talking about a class as restricted as the regular polytopes).
There are 15 known edge-transitive 4-polytopes (and an infinite family of duoprisms\footnote{The $(n,m)$-duoprism is the cartesian product of a regular $n$-gon and a regular $m$-gon. Those are edge-transitive if and only of $n=m$. Technically, the 4-cube is the $(4,4)$-duoprism but is usually not counted as such, because of its exceptionally large symmetry group.}), but already here, no classification is known.
It is known that the number of irreducible\footnote{Being not the cartesian product of lower dimensional edge-transitive polytopes.} edge-transitive polytopes grows at least linearly with the number of dimensions.
For example, there are $\lfloor d/2\rfloor$ \emph{hyper-simplices} in dimension $d$.
These are edge-transitive (even distance-transitive, see \cref{sec:distance_transitive}).

It is the hope of the author, that the classification of the edge-transitive polytopes can be obtained using their spectral properties.
Their classification can now be stated purely as a problem in spectral graph theory: 
the classification of the edge-transitive polytopes (in dimension $d\ge 4$) is equivalent to the classification of $\theta_2$-spectral edge-transitive graphs, and since \cref{res:spectral_2}, we have a completely graph theoretic characterization of spectral graphs. %, and the classification can be attacked from this perspective.

%To see this, note the following equivalent definition of \enquote{spectral graph}, not making explicit reference to polytopes:

%\begin{lemma}
%\label{res:spectral_2}
%A graph is $\theta$-spectral if and only if it satisfies both:
%%
%\begin{myenumerate}
%	\item for every vertrex $i\in V$ there is a $\theta$-eigenvector $u=(u_1,...,u_n)\in\RR^n$ with
%%
%$$\Argmax_{k\in V} u_k = \{i\}.$$
%	\item for all distinct vertices $i,j\in V$ we have that $ij\in E$ is an edge of $G$ if and only if there is a $\theta$-eigenvector $u=(u_1,...,u_n)\in\RR^n$ with
%%
%$$\Argmax_{k\in V} u_k = \{i,j\}.$$
%\end{myenumerate}
%\end{lemma}

%\begin{definition}
%\label{def:spectral_2}
%A graph $G$ is said to be \emph{$\theta$-spectral} for some eigenvalue $\theta\in\Spec(G)$, if $ij\in E$ is an edge if and only if there is a $\theta$-eigenvector $u=(u_1,...,u_n)\in\RR^n$ with
%%
%$$\Argmax_{k\in V} u_k = \{i,j\}.$$
%\end{definition}

%\begin{corollary}
%Let $G$ be an edge-transitive $\theta_k$-spectral graph.
%\begin{myenumerate}
%	\item If $G$ is vertex-transitive, then $k=2$.
%	\item If $G$ is not vertex-transitive, then $k=2$ and $G$ is the edge-graph of the rhombic dodecahedron (see \cref{fig:edge_transitive}).
%\end{myenumerate}
%\end{corollary}

\begin{theorem}
\label{res:edge_transitive_spectral_graph}
Let $G$ be an edge-transitive graph. If $G$ is $\theta_k$-spectral, then
\begin{myenumerate}
	\item $k=2$, and
	\item if $G$ is \ul{not} vertex-transitive, then $G$ is the edge-graph of the rhombic dodecahedron (see \cref{fig:edge_transitive}).
\end{myenumerate}
\begin{proof}
We first prove $(ii)$.
As shown in \cite{winter2020polytopes} all edge-transitive polytopes in dimen\-sion $d\ge 4$ are vertex-transitive.
If $G$ is edge-transitive, not vertex-transitive and $\theta_k$-spectral, then its $\theta_k$-eigenpolytope is also edge-transitive but not vertex-transitive, hence of dimension $d\le 3$.
One checks that the 2-dimensional spectral polytopes are regular polygons, hence vertex-transitive.
The remaining polytopes are polyhedra, and we mentioned in \cref{rem:edge_not_vertex} that among these, only the rhombic dodecahedron is spectral, in fact $\theta_2$-spectral.
This proves $(ii)$.

Equivalently, if $G$ is vertex- and edge-transitive, then so is its eigenpolytope.
By \cref{res:vertex_edge_transitive_cor} this is a $\theta_2$-eigenpolytope.
Together with part $(ii)$, we find $k=2$~in~all cases, which proves $(i)$.
\end{proof}
\end{theorem}

%\begin{corollary}
%Let $G$ be an edge-transitive graph. If $G$ is $\theta_k$-spectral, then $k=2$. If in addition it is \emph{not} vertex-transitive, then $G$ is the edge-graph of the rhombic dodecahedron (see \cref{fig:edge_transitive}).
%\end{corollary}

\subsection{Arc- and half-transitive polytopes}
\label{sec:arc_transitive}

In a graph or polytope, an \emph{arc} is~an incident vertex-edge-pair.
A graph or polytope is called \emph{arc-transitive} if its symmetry group acts transitively on the arcs.
Being arc-transitive implies both, being vertex-transitive, and being edge-transitive.
In addition to that, in an arc-transitive graph, every edge can be mapped, not only onto every other edge, but also onto itself with flipped orientation.

There exist graphs that are simul\-taneously vertex- and edge-transitive, but not arc-transitive.
Those are called \emph{half-transitive} graphs, and are comparatively rare.
The smallest one has $27$ vertices and is known as the \emph{Holt graph} (see \cite{bouwer1970vertex,holt1981graph}).

For polytopes on the other hand, it is unknown whether there eixsts a distinction being arc-transitive and being simultaneously vertex- and edge-transitive.
No \emph{half-transitive polytope} is known.
Because of \cref{res:edge_vertex_transitive} $(i)$, we know that the edge-graph of a half-transitive polytope must itself be half-transitive.
Since such graphs are rare, the existence of half-transitive polytopes seems unlikely.

\begin{example}
The Holt graph is not the edge-graph of a half-transitive polytope: the Holt graph is of degree four, and its second-largest eigenvalue is of multiplicity six, giving rise to a 6-dimensional $\theta_2$-eigenpolytope. 
But a 6-dimensional polytope must have an edge-graph of degree at least six, and so the Holt graph is not spectral.
\end{example}

The lack of examples of half-transitive polytopes means that all known edge-transitive polytopes in dimension $d\ge 4$ are in fact arc-transitive.
Likewise, a~classification of arc-transitive polytopes is not known.

\subsection{Distance-transitive polytopes}
\label{sec:distance_transitive}

Our previous results about edge-transitive polytopes already allow for a complete classification of a particular subclass, namely, the \emph{distance-transitive polytopes}, thereby also providing a list of examples of edge-transitive polytopes in higher dimensions.

The distance-transitive symmetry is usually only considered for graphs, and the distance-transitive graphs form a subclass of the distance-regular graphs.
The usual reference for these is the classic monograph by Brouwer, Cohen and Neumaier \cite{brouwer1989distance}.

For any two vertices $i,j\in V$ of a graph $G$, let $\dist(i,j)$ denote the graph-theoretic \emph{distance} between those vertices, that is, the length of the shortest path connecting them.
The \emph{diameter} $\diam(G)$ of $G$ is the largest distance between any two vertices in $G$.

\begin{definition}
A graph is called \emph{distance-transitive} if $\Aut(G)$ acts transitively on each of the sets
$$D_{\delta}:=\{(i,j)\in V\times V \mid \dist(i,j)=\delta\},\quad\text{for all $\delta\in\{0,...,\diam(G)\}$}.$$

Analogously, a polytope $P\subset\RR^d$ is said to be \emph{distance-transitive}, if its Euclidean symmetry group $\Aut(P)$ acts transitively on each of the sets 
$$D_{\delta}:=\{(v_i,v_j)\in \F_0(P)\times \F_0(P) \mid \dist(i,j)=\delta\},\quad\text{for all $\delta\in\{0,...,\diam(G_P)\}$}.$$
Note that the distance between the vertices is still measured along the edge-graph rather than via the Euclidean distance.
\end{definition}

Being arc-transitive is equivalent to being transitive on the set $D_1$.
Hence, distance-transitivity implies arc-transitivity, thus edge-transitivity.

By our considerations in the previous sections, we know that the \mbox{classification} of distance-transitive polytopes is equivalent~to the classification of the $\theta_2$-spectral distance-transitive graphs.
Those where classified by Godsil (see \cref{res:spectral_distance_regular_graphs}).

%By \cref{res:edge_vertex_transitive}, we known that the edge-graph $G_P$ of a distance-transitive polytope must be distance-transitive, and that the polytope is the $\theta_2$-eigenpolytope of $G_P$.
%%
%But Godsil already gave a classification of  $\theta_2$-spectral distance-regular graphs (see \cref{res:spectral_distance_regular_graphs}; all of which~turned out to be distance-transitive) allows a complete classification of distance-tran\-sitive polytopes.

In the following theorem we translated each such $\theta_2$-spectral distance-transitive graph into its respective eigenpolytope.
This gives a complete classification of the distance-transitive polytopes.

\begin{theorem}\label{res:distance_transitive_classification}
If $P\subset\RR^d$ is distance-transitive, then it is one of the following:
\begin{enumerate}[label=$(\text{\roman*}\,)$]
	\item a regular polygon $(d=2)$,
	\item the regular dodecahedron $(d=3)$,
	\item the regular icosahedron $(d=3)$,
	\item a cross-polytopes, that is, $\conv\{\pm e_1,...,\pm e_d\}$ where $\{e_1,...,e_d\}\subset\RR^d$ is the standard basis of $\RR^d$,
	\item a hyper-simplex $\Delta(d,k)$, that is, the convex hull of all vectors $v\in\{0,1\}^{d+1}$ with exactly $k$ 1-entries,
	\item a cartesian power of a regular simplex (also known as the Hamming polytopes; this includes regular simplices and hypercubes),
	\item a demi-cube, that is, the convex hull of all vectors $v\in\{-1,1\}^d$ with~an~even number of 1-entries,
	\item the $2_{21}$-polytope, also called Gosset-polytope $(d=6)$,
	\item the $3_{21}$-polytope, also called Schläfli-polytope $(d=7)$.
\end{enumerate}
The ordering of the polytopes in this list agrees with the ordering of graphs in the list in \cref{res:spectral_distance_regular_graphs}.
The latter two polytopes where first constructed by Gosset in \cite{gosset1900regular}.
\end{theorem}

We observe that the list in \cref{res:distance_transitive_classification} contains many polytopes that are not regular, and contains all regular polytopes excluding the 4-dimensional exceptions, the 24-cell, 120-cell and 600-cell.
The distance-transitive polytopes thus form a distinct class of remarkably symmetric polytopes which is not immediately related to the class of regular polytopes.

Another noteworthy observation is that all the distance-transitive polytopes are \emph{Wythoffian poly\-topes}, that is, they are orbit polytopes of finite reflection groups.
\Cref{fig:distance_transitive_Coxeter} shows the Coxeter-Dynkin diagrams of these polytopes.

\begin{figure}
\centering
\includegraphics[width=0.9\textwidth]{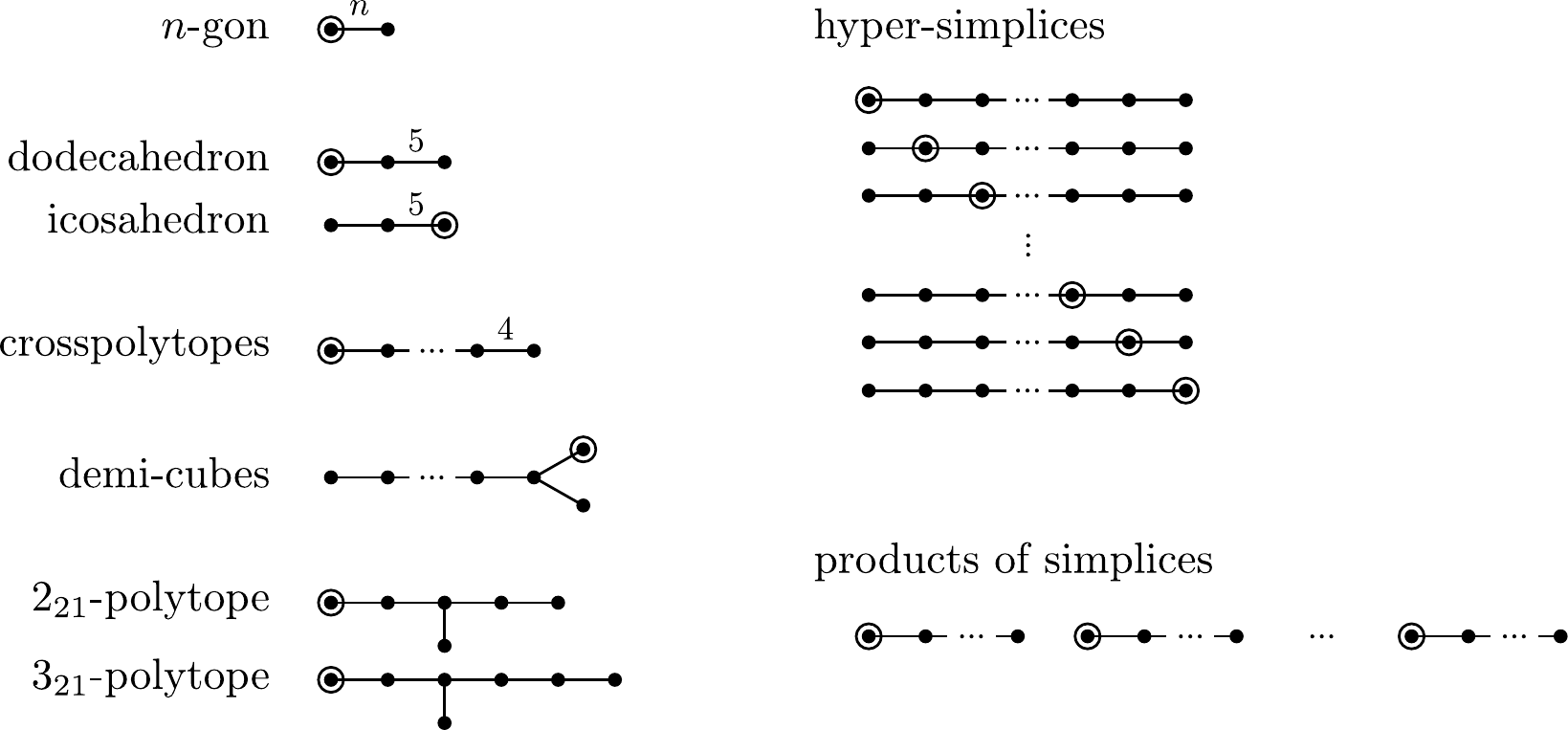}
\caption{Coxeter-Dynkin diagrams of distance-transitive polytopes.}
\label{fig:distance_transitive_Coxeter}
\end{figure}

\tempnewpage
\section{Conclusion and open questions}
\label{sec:future}

In this paper we have studied \emph{eigenpolytopes} and \emph{spectral polytopes}.
The former are polytopes constructed from a graph and one of its eigenvalues.
A polytope is spectral if it is the eigenpolytopes of its edge-graph.
These are of interest because spectral graph theory then ensures a strong interplay between the combinatorial properties of the edge-graph and the geometric properties of the polytope.

The study of eigenpolytopes and spectral polytopes has left us with many open questions.
Most notably, how to detect spectral polytopes purely from their geome\-try.
We introduced a tool (\cref{res:implies_spectral}), which was sufficient to proof that (most) edge-transitive polytopes are spectral.
We do not know how much more general it can be applied.

\begin{question}
\label{q:characterization}
Does \cref{res:implies_spectral} already characterize $\theta_2$-spectral polytopes (or even spectral polytopes in general)?
\end{question}

If the answer is affirmative, this would provide a geometric characterization of polytopes that are otherwise defined purely in terms of spectral graph theory.
The result of Izmestiev suggests that polytopes with sufficiently regular geometry are $\theta_2$-spectral: the entry of the matrix $X$ in \cref{res:izmestiev} at index $ij\in E$ can be~expressed as
$$X_{ij}=\frac{\vol(\sigma_i\cap\sigma_j)}{\|v_i\|\|v_j\|\sin\angle(v_i,v_j)},$$
where $\sigma_i$ and $\sigma_j$ are the facets of the polar dual $P^\circ$ that correspond to the vertices $v_i,v_j\in\F_0(P)$.
Because of this formula, it might be actually easier to classify the polar duals of $\theta_2$-spectral polytopes.

An affirmative answer to \cref{q:characterization} would also mean a negative answer to the following:

\begin{question}
\label{q:not_theta_2}
Is there a $\theta_k$-spectral polytope/graph for some $k\not=2$?
\end{question}

The answer is known to be negative for edge-transitive polytopes/graphs (see \cref{res:edge_transitive_spectral_graph}), but unknown in general.

The second-largest eigenvalue $\theta_2$ is special for other reasons too.
Even if a graph is not $\theta_2$-spectral, it seems to still imprint its adjacency information onto the edge-graph of its $\theta_2$-eigenpolytope.

\begin{question}
\label{q:realizing_edges}
Given an edge $ij\in E$ of $G$, if $v_i$ and $v_j$ (as defined in \cref{def:eigenpolytope}) are distinct vertices of the $\theta_2$-eigenpolytope $P_G(\theta_2)$, is then also $\conv\{v_i,v_j\}$ an edge of $P_G(\theta_2)$?
\end{question}

This was proven for distance-regular graphs in \cite{godsil1998eigenpolytopes}, and is not necessarily true for eigenvalues other than $\theta_2$.

All known spectral polytopes are exceptionally symmetric.
It is unclear whether this is true in general.

\begin{question}
\label{q:trivial_symmetry}
Are there spectral polytopes with trivial symmetry group?
\end{question}

An example for \cref{q:trivial_symmetry} must be asymmetric, yet with a reasonably large eigenspaces.
Such graphs exist among the distance-regular graphs, but all spectral distance-regular graphs were determined in \cite{godsil1998eigenpolytopes} (see also \cref{res:spectral_distance_regular_graphs}) and turned out to be distance-transitive, \ie\ highly symmetric.

A clear connection between being spectral and being symmetric is missing.
To emphasize our ignorance, we ask the following:

\begin{question}
\label{q:spectral_non_vertex_transitive}
Can we find more spectral polytopes that are \emph{not vertex-transitive}?
What characterizes them?
\end{question}

The single known spectral polytope that is \emph{not} vertex-transitive is the \emph{rhombic dodecahedron} (see \cref{fig:edge_transitive}).
The fact that it is spectral appears purely accidental, as there seems to be no reason for it to be spectral, except that we can explicitly check that it is. For comparison, the highly related \emph{rhombic triacontahedron} is not spectral.

On the other hand, vertex-transitive spectral polytopes might be quite common.

\begin{question}
\label{q:specific_instance}
Let $P\subset\RR^d$ be a polytope with the following properties:
\begin{myenumerate}
	\item $P$ is vertex-transitive,
	\item $P$ realizes all the symmetries of its edge-graph, and
	\item $\Aut(P)$ is irreducible.
\end{myenumerate}
Is $P$ (combinatorially equivalent to) a spectral polytope? %Or at least an eigenpolytope?
\end{question}

No condition in \cref{q:specific_instance} can be dropped.
If we drop vertex-transitivity,~we could take some polytope whose edge-graph has trivial symmetry and only small eigenspaces. Dropping $(ii)$ leaves vertex-transitive neighborly polytopes, for which we know that these are mostly not spectral (except for the simplex).
Dropping $(iii)$ leaves us with the prisms and anti-prisms, the eigenspaces of their edge-graphs are rarely of dimension greater than two.

%\cref{q:specific_instance} has equally interesting consequences if it is only true for the eigenpolytope case.
%For example, several polytopes with relevance in integer optimization seem to be eigenpolytopes, \eg\ the Birkhoff polytope or the traveling salesman polytopes.
%This would allow a formulation of the respective optimization problem in the language of spectral graph theory.
%Also, this gives a way to study the symmetry groups of these polytopes (which are oftem much larger than expected from their definition).

Finally, we wonder whether these spectral techniques can be any help in classifying the edge-transitive polytopes.

\begin{question}
Can we classify the edge-transitive graphs that are spectral, and~by this, the edge-transitive polytopes?
\end{question}

\begin{question}
Can the existence of half-transitive polytopes be excluded by using spectral graph theory (see \cref{sec:arc_transitive})?
\end{question}

\par\bigskip
\parindent 0pt
\textbf{Acknowledgements.} The author gratefully acknowledges the support by the funding of the European Union and the Free State of Saxony (ESF).

%%%%%%%%%%%%%%%%%%%%%%%%%%%%%%%%%%%%%%%%%%%%%%%%%%%%%%%%%%%%%%%%%%%%%%%%%%%%%%%%%%

\bibliographystyle{abbrv}
\bibliography{literature}

\begin{thebibliography}{10}

\bibitem{petrovMO}
Math{O}verflow.
\newblock \url{https://mathoverflow.net/a/325073/108884}.
\newblock Accessed: 2020-08-29.

\bibitem{aigner2010proofs}
M.~Aigner, G.~M. Ziegler, K.~H. Hofmann, and P.~Erdos.
\newblock {\em Proofs from the Book}, volume 274.
\newblock Springer, 2010.

\bibitem{blueSpectral}
``Blue''.
\newblock Spectral realizations of graphs.
\newblock webite:
  ``daylateanddollarshort.com/mathdocs/Spectral-Realizations-of-Graphs.pdf''
  (August 2020).

\bibitem{bokowski1984combinatorial}
J.~Bokowski, G.~Ewald, and P.~Kleinschmidt.
\newblock On combinatorial and affine automorphisms of polytopes.
\newblock {\em Israel Journal of Mathematics}, 47(2-3):123--130, 1984.

\bibitem{bouwer1970vertex}
I.~Bouwer.
\newblock Vertex and edge transitive, but not 1-transitive, graphs.
\newblock {\em Canadian Mathematical Bulletin}, 13(2):231--237, 1970.

\bibitem{brouwer1989distance}
A.~Brouwer, A.~Cohen, and A.~Neumaier.
\newblock Distance-regular graphs. 1989.
\newblock {\em Ergeb. Math. Grenzgeb.(3)}, 1989.

\bibitem{gevay2002perfect}
G.~G{\'e}vay.
\newblock On perfect 4-polytopes.
\newblock {\em Beitr{\"a}ge zur Algebra und Geometrie}, 43(1):243--259, 2002.

\bibitem{godsil1995euclidean}
C.~Godsil.
\newblock Euclidean geometry of distance regular graphs.
\newblock {\em London Mathematical Society Lecture Note Series}, pages 1--24,
  1995.

\bibitem{godsil1978graphs}
C.~D. Godsil.
\newblock Graphs, groups and polytopes.
\newblock In {\em Combinatorial Mathematics}, pages 157--164. Springer, 1978.

\bibitem{godsil1998eigenpolytopes}
C.~D. Godsil.
\newblock Eigenpolytopes of distance regular graphs.
\newblock {\em Canadian Journal of Mathematics}, 50(4):739--755, 1998.

\bibitem{gosset1900regular}
T.~Gosset.
\newblock On the regular and semi-regular figures in space of $n$ dimensions.
\newblock {\em Messenger of Mathematics}, 29:43--48, 1900.

\bibitem{holt1981graph}
D.~F. Holt.
\newblock A graph which is edge transitive but not arc transitive.
\newblock {\em Journal of Graph Theory}, 5(2):201--204, 1981.

\bibitem{izmestiev2010colin}
I.~Izmestiev.
\newblock The colin de verdiere number and graphs of polytopes.
\newblock {\em Israel Journal of Mathematics}, 178(1):427--444, 2010.

\bibitem{licata1986surprising}
C.~Licata and D.~L. Powers.
\newblock A surprising property of some regular polytopes.
\newblock Technical report, CLARKSON UNIV POTSDAM NY DEPT OF MATHEMATICS AND
  COMPUTER SCIENCE, 1986.

\bibitem{mohri1997theta_1}
T.~Mohri.
\newblock The $\theta_1$-eigenpolytopes of the hamming graphs (groups and
  combinatorics).
\newblock 1997.

\bibitem{padrol2010graph}
A.~Padrol~Sureda and J.~Pfeifle.
\newblock Graph operations and laplacian eigenpolytopes.
\newblock In {\em VII Jornadas de Matem{\'a}tica Discreta y Algor{\'\i}tmica},
  pages 505--516, 2010.

\bibitem{pak2010lectures}
I.~Pak.
\newblock Lectures on discrete and polyhedral geometry.
\newblock {\em Manuscript (http://www. math. ucla. edu/\~{} pak/book. htm)},
  2010.

\bibitem{powers1986petersen}
D.~L. Powers.
\newblock The petersen polytopes.
\newblock Technical report, CLARKSON UNIV POTSDAM NY DEPT OF MATHEMATICS AND
  COMPUTER SCIENCE, 1986.

\bibitem{powers1988eigenvectors}
D.~L. Powers.
\newblock Eigenvectors of distance-regular graphs.
\newblock {\em SIAM journal on matrix analysis and applications},
  9(3):399--407, 1988.

\bibitem{rooney2014spectral}
B.~Rooney.
\newblock Spectral aspects of cocliques in graphs.
\newblock 2014.

\bibitem{winter2019geometry}
M.~Winter.
\newblock Geometry and topology of symmetric point arrangements.
\newblock {\em arXiv preprint arXiv:1907.11120}, 2019.

\bibitem{winter2020polytopes}
M.~Winter.
\newblock On polytopes that are edge-transitive but not vertex-transitive,
  2020.

\bibitem{winter2020symmetric}
M.~Winter.
\newblock Symmetric and spectral realizations of highly symmetric graphs, 2020.

\end{thebibliography}

\newpage

\appendix

\section{Implementation in Mathematica}
\label{sec:appendix_mathematica}

The following short Mathematica script takes as input a graph $G$ (in the example below, this is the edge-graph of the dodecahedron), and an index $k$ of an eigenvalue.
It then compute the $v_i$ (or \texttt{vert} in the code), \ie\ the vertex-coordinates of the $\theta_k$-eigenpolytope.
If the dimension turns out to be appropriate, the spectral embedding of the graph, as well as the eigenpolytope are plotted.

\vspace{0.5em}
\begin{lstlisting}
(* Input: 
  * the graph G, and
  * the index k of an eigenvalue (k = 1 being the largest eigenvalue).
*)
G = GraphData["DodecahedralGraph"];
k = 2;

(* Computation of vertex coordinates 'vert' *)
n = VertexCount[G];
A = AdjacencyMatrix[G];
eval = Tally[Sort@Eigenvalues[A//N], Round[#1-#2,0.00001]==0 &];
d = eval[[-k,2]]; (* dimension of the eigenpolytope *)
vert = Transpose@Orthogonalize@
  NullSpace[eval[[-k,1]] * IdentityMatrix[n] - A];

(* Output: 
  * the graph G, 
  * its eigenvalues with multiplicities, 
  * the spectral embedding, and
  * its convex hull (the eigenpolytope).
*)
G
Grid[Join[{{$\theta$,"mult"}}, eval], Frame$\to$All]
Which[
  d<2 , Print["Dimension too low, no plot generated."],
  d==2, GraphPlot[G, VertexCoordinates$\to$vert],
  d==3, GraphPlot3D[G, VertexCoordinates$\to$vert,
  d>3 , Print["Dimension too high, 3-dimensional projection is plotted."];
    GraphPlot3D[G, VertexCoordinates$\to$vert[[;;,1;;3]] ]
]
If[d==2 || d==3,
  Region`Mesh`MergeCells[ConvexHullMesh[vert]]
]
\end{lstlisting}

\end{document}